\documentclass[graybox]{svmult}

\usepackage{mathptmx}       % selects Times Roman as basic font
\usepackage{helvet}         % selects Helvetica as sans-serif font
\usepackage{courier}        % selects Courier as typewriter font
\usepackage{makeidx}         % allows index generation
\usepackage{graphicx}        % standard LaTeX graphics tool
\usepackage{multicol}        % used for the two-column index
\usepackage[bottom]{footmisc}% places footnotes at page bottom

\usepackage{amsmath}%
\usepackage{amsfonts}%
\usepackage{amssymb}%
\newtheorem{notation}{Notation}

\newcommand{\be}{\begin{equation}}
\newcommand{\ee}{\end{equation}}
\newcommand{\h}{\lim_{h\rightarrow0}}
\newcommand{\N}{\mathbb{N}}
\newcommand{\R}{\mathbb{R}}

\makeindex             % used for the subject index
\begin{document}
\title{Evolutionary game of coalition building under external pressure}
\author{Alekos Cecchin and Vassili N. Kolokoltsov}
\institute{Alekos Cecchin 
\at Department of Mathematics, University of Padua, 
 Via Trieste 63, Padova, Italy,
\newline
  \email{acecchin@math.unipd.it}
\and Vassili N.Kolokoltsov \at Department of Statistics, University of Warwick,
Coventry, CV4 7AL, UK
\newline
 \email{v.kolokoltsov@warwick.ac.uk}}
%\address[Alekos Cecchin]
%{Department of Mathematics, University of Padua, \newline
%\indent Via Trieste 63, Padova, Italy}%
%\email{acecchin@math.unipd.it}
%\urladdr{http://www.authorone.uni-aone.de}
%\author{}
%\address[Vassili N. Kolokoltsov]{Department of Statistics, University of Warwick, \newline
%\indent Coventry, CV4 7AL, UK}
%\email{v.kolokoltsov@warwick.ac.uk}
%\urladdr{http://www.authortwo.uni-atwo.hu}
%\author{Author Three}
%\address[A. Three]{Author Three address, line 1\newline%
%\indent Author Three address, line 2}
%\email[A.~Three]{author-three@authorthree-inst.edu}%
%\urladdr{http://www.authorthree.uni-athree.edu}
%\thanks{Work (partially) supported by the PhD programme in Mathematical Sciences, Dipartimento di Matematica, Universit\`a di Padova (Italy) and Progetto Dottorati - Fondazione Cassa di Risparmio di Padova e Rovigo}
%\thanks{Thanks for Author Two.}
%\date{\today}
%\subjclass{Primary 05C38, 15A15; Secondary 05A15, 15A18} %
%\keywords{Fragmentation-coagulation, merging and splitting, evolutionary coalition formation, Markov decision process, major agent, mean field limit}
%\dedicatory{Dedicated to Professor XY on the occasion of his seventieth birthday.}
\maketitle

\abstract{
We study the fragmentation-coagulation, or merging and splitting, model as introduced in \cite{V}, where $N$ small players can form coalitions to resist to the pressure exerted by the principal. It is a Markov chain in continuous time and the players have a common reward to optimize. 
 We study the behavior as $N$ grows  and show that the problem converges to a (one player) deterministic optimization problem in continuous time, in the infinite dimensional state space $\ell^1$.
We apply the method developed in \cite{M}, adapting it to our  different framework. We use tools involving dynamics in $\ell^1$, generators of Markov processes, martingale problems and coupling of Markov chains.
}

\abstract*{
We study the fragmentation-coagulation, or merging and splitting, model as introduced in \cite{V}, where $N$ small players can form coalitions to resist to the pressure exerted by the principal. It is a Markov chain in continuous time and the players have a common reward to optimize. 
 We study the behavior as $N$ grows  and show that the problem converges to a (one player) deterministic optimization problem in continuous time, in the infinite dimensional state space $\ell^1$.
We apply the method developed in \cite{M}, adapting it to our  different framework. We use tools involving dynamics in $\ell^1$, generators of Markov processes, martingale problems and coupling of Markov chains.
}
%% Based on a TeXnicCenter-Template by Gyorgy SZEIDL.
%%%%%%%%%%%%%%%%%%%%%%%%%%%%%%%%%%%%%%%%%%%%%%%%%%%%%%%%%%%%%

\section{Introduction}

%Mean field game theory is devoted to the analysis of differential games with a very large
%number of small players. By small player, we mean a player who has very little influence
%on the overall system. This theory has been recently developed in 2006 by J.-M. Lasry and P.-L. Lions
%in a series of papers (see e.g. \cite{43}) and presented though several lectures of P.-L. Lions at the
%College de France. Its name comes from the analogy with the mean field models in mathematical
%physics which analyzes the behavior of many identical particles.
%The mean field games theory seems also particularly adapted to modeling 
%problems in economics: see e.g. Lasry, Lions, Gueant \cite{46}.

In this paper  we study dynamic optimization problems on Markov decision processes composed of
a large number of interacting agents, in particular we investigate the so called fragmentation-coagulation, or merging and splitting, model.
Our aim is to analyze its limit as the number of players tends to infinity.

%In this project we study a particular type of models, the so called fragmentation coagulation models.
%So far we talked about small players that change their state in discrete time.
Following Kolokoltsov \cite{V} we describe a model  which is  a Markov chain in continuous time. A natural reaction of the society
of small players to the pressure exerted by the principal can be executed by forming stable
groups that can confront this pressure in an effective manner (but possibly imposing
certain obligatory regulations for the members of the group). Analysis of such possibility
leads one naturally to models of mean field enhanced coagulation processes under external
pressure. The major player can change her strategy only in discrete deterministic time.
%The goal of the controller is to maximize one common reward.

Coagulation fragmentation processes are well studied in statistical physics, see
e. g. \cite{71}. In particular, general mass exchange processes, that in our social environment
become general coalition forming processes preserving the total number of participants,
were analyzed in \cite{48} and \cite{50} with their law of large number limits for discrete and general
state spaces. In the same way problems in economics, like merging banks
or firms on the market, were studied in \cite{75} and \cite{78}. While an application to 
 scientific citation networks or the network of internet links is discussed in \cite{62}. 
Some simple situation of nonlinear Markov games on a finite state space was analyzed in \cite{Kol}, proving the convergence of Nash-equilibria for finite games to equilibria of a limiting deterministic differential game.

%Models
%of growth are known to lead to power laws in equilibrium, which are verified in a variety of
%real life processes, see e.g. [78] for a general overview and [76] for particular applications in
%crime rates. 

Very recently, several authors have studied games  of coalition formation. 
A notion of core equilibrium in proposed in \cite{I}, and found via a fixed point method. 
An application to contracts and networks is analyzed in \cite{jkk}. 
A study of the incentives offered by the government to municipalities
to merge into larger groups is provided in \cite{w}.
 Players preferences over winning coalitions are derived  by applying
strongly monotonic power indices on the game in \cite{karos}, where  the author also investigate
whether there are core stable coalitions . 
An application of systems of coalition formation to the climate change problem is discussed in \cite{finus}, where also numerical simulation are performed.

%There are other two complications on this model with respect to the work of Le boudec et al.
Here we are interested in the response of such systems to external parameters
that may be set by the principal who has her own
agenda. Thus we add to the analysis a major player fitting
the model to a more general framework.
There are two main difficulties in studying this model. Firstly the total number of  coalitions is not constant in time, as they can merge or split. Secondly the dynamics both of the system of small players and of the limiting system are supposed to lie on the infinite dimensional space $\ell^1$, which can be viewed also as a space of measures, instead of a fixed $\R^d$. In fact the dimension of the state space for the system of coalitions grows as the number $N$ of small players tends to infinity. If the system is in the state $x\in \ell^1$ then $x_k = h n_k$ where $n_k$ is the number of coalitions of size $k$ and $h$ is a suitable parameter depending on $N$, for instance the inverse of the initial number of coalitions. 

%required to ensure the convergence. So we need to find a compact set stable under all the considered dynamics in $\l$.

Our main result is to show that this problem converges, as $N$ grows, to  a one player deterministic optimization problem in continuous time, in the infinite dimensional state space $\ell^1$, the so called mean field limit. We prove convergence of the value functions and provide also an approximated optimal policy for the system of small players. Such optimal policy is usually found by using dynamic programming for the finite horizon case, but this approach suffer
from the curse of dimensionality, which makes them impractical when the state space is too large. 
 Solving the HJB equation for the limiting system numerically is sometimes rather easy. It provides a deterministic optimal policy whose reward is remarkably close to the optimal reward.

We apply the method developed by Gast, Gaujal and Le Boudec in \cite{M}, where the authors obtained the same kind of results , but in a different setting. They consider discrete time Markov chains as prelimit systems whose state space is finite and fixed. 
%However their method of proof can still work in our model.
 %Their techniques  are rather different from the other works on this subject. 
Their proofs are  in line
with classic mean field arguments and use stochastic approximation techniques.
 Moreover their approach is  algorithmic: they construct two intermediate systems: one with a finite number of
objects controlled by a limit policy and one with a limit system controlled by a stochastic policy
induced by the finite system.

%In the beginning, consider a system of N small players evolving in a common environment, which is described by a Markov chain in discrete time. At each time step, objects
%change their state randomly, depending on the
%number of objects in each state as well as on the decisions of a centralized controller. Our goal is
%to study the behavior of the controlled system when N becomes large, the so called mean field limit. This is the aim of \cite{M} where the authors show the convergence under general assumptions. 

%Several papers investigate the asymptotic behavior of such systems, but without controllers.
Several papers in the literature are concerned with the problem of mixing the limiting behavior of
a large number of objects with optimization.
%Benaim in \cite{2} studies similar systems without contollers and shows that under mild conditions, as N grows, the system
%converges to a deterministic limit. 
In \cite{6}, the value function of the Markov decision process (MDP) is approximated by a linearly
parametrized class of functions and a 
fluid approximation of the MDP is used. It is shown
that a solution of the HJB equation is a value function for a modification of the original MDP
problem. In \cite{25}, the curse of dimensionality of dynamic programming is circumvented by
approximating the value function by linear regression. In \cite{M} they use instead a mean field limit
approximation and prove asymptotic optimality in N of limit policy.
Actually, most of the papers dealing with mean field limits of optimization problems over
large systems are set in a game theory framework, leading to the concept of mean field games, introduced by Lasry and Lions \cite{43}
 and
P.E. Caines, M. Huang and R.P. Malham\'e \cite{HMC}.

Notice finally that in this paper we analyze only a preliminary step for a full game setting with major and minor players, namely the response of the minor players to the action of the major one. The full analysis (not developed here) would include the reaction of the major on the behavior of the minor players and the search for the corresponding  equilibrium. However, this development does not seem to present serious difficulty, since our analysis reduces it effectively to a two-player game: the major and the pool of small players.

\subsection*{Contribution and structure of the paper}

In \cite{V} Kolokoltsov shows the convergence of the optimization problems related to the system of small players to an optimization problem in discrete time for the limiting system (this will be similar to theorem \ref{teo5}). His proof is based on an argument that is focused on the generators of the Markov chains and shows their convergence. In this paper we want to show the convergence to an optimization problem in continuous time, so we apply the ideas from \cite{M} where they used a completely different argument for the proof, focusing on trajectories and constructing two auxiliary systems.

In section 2 we describe properly the fragmentation coagulation model starting from \cite{V}, the limiting system and the related optimization problems. We  define the state space where all the dynamics considered lie, which is a compact set $S\subset\ell^1$, and in the end we state the assumptions we need to obtain the convergence. In section 3 we present our main results and define the two auxiliary systems. 
Then we show how to construct an approximated optimal policy starting from an optimal action function for the mean field limit.
Moreover we consider a class of applications in which a particular choice of the rate functions allows to reduce the limiting problem to an optimization problem in one dimension, providing an explicit solution in a simplified case and a more effective numerical scheme in general. Finally in section 4 we complete the proofs, showing that the general requirements for convergence expressed in \cite{M} can fit to our model, with some modification. We use theorems about semigroups and generators of Markov processes and related martingale problems. Moreover we apply the notion of coupling of  Markov chains and also a particular Markovian coupling.

%Our contribution is then to show that all the general assumptions required in \cite{M} are satisfied in our model. There is some difference between the models, as we tried to explain above; however we show that our model can fit to the hypothesis of the cited article, with a slight modification of them which is still working.
%We use many mathematical tools to achieve this goal. Firstly we need some result about functions and dynamics in spaces of measures, like $\l$. Then we use the theorems about semigroups and generators of Markov processes and related martingale problems. Moreover we need the notion of coupling of measures and of Markov chains and also a particular Markovian coupling.

%\subsection*{}

%In the 2.4 we state the general hypothesis assumed in the article \cite{M} to get the convergence and then in the section 2.5 we build the two auxiliary systems. Hence in section 3 we show how those assumptions are satisfied in our model, which is our main contribution. Then, for the sake of completeness, in the third chapter we state the results found in \cite{M} and write their proofs. We add some comment and adjust some proof to our slightly different framework. 

%\subsection*{Acknowledgements}

%This dissertation has been mostly prepared at the University of Warwick under the supervision of Professor Vassili Kolokoltsov. I want to thank him for the availability and for all the useful suggestions and hints.

%\part{The First Part}

\section{Model and assumptions}

\subsection{The space $B_+(L,R]$}

We denote, as usual, the space of measures
\begin{equation}
\ell^1(\mathbb{N}) :=\left\{x =(x_1,x_2, \ldots) : x_k \in \mathbb{R} , \quad \sum_k |x_k|<\infty\right\} .
\end{equation}
Denote by $\ell^1_+$  the space of positive measures on $\N$:
$\ell^1_+(\N) := \left\{x \in \ell^1 : x_k\geq0 \right\}$.
The usual norm in $\ell^1$ is 
$||x||_{\ell^1}:= \sum_k |x_k|$.
%We have to recall what is a Lyapunov functions.

%\begin{definition}
%A function $L:\mathbb{N}\longrightarrow \mathbb{R}$ is a \emph{Lyapunov function} if it is bounded below by a positive constant. 
%and such that $\lim_{k\rightarrow\infty} L(k) =\infty$.
%\end{definition}

 %We recall that $\ell^1$ is the space of finite signed measures on $\mathbb{N}$, 
%$$\ell^1(\mathbb{N}) =\left\{x =(x_1,x_2, \ldots) : x_k \in \mathbb{R} , \quad \sum_k |x_k|<\infty\right\} ,$$
%and it is contained in $\ell^2(\N)$;
Let $L:\mathbb{N}\longrightarrow \mathbb{R}$ be the identity function, which means $L(k) = k$. 
We define a new norm
$ ||x||_{\ell^1(L)} := \sum_k L(k) |x_k|$,
so that we can consider the subset of  $\ell^1$
$$ \ell^1(L) := \left\{x\in\ell^1 : ||x||_{\ell^1(L)} <\infty\right\}$$
which is a Banach space equipped with this norm. 
%In particular then $\ell^1=\ell^1(\textbf{1})$, where $\textbf{1}$ denotes the function that is constantly 1 for any $k\in\N$.

Let us denote by $B(L,R)$ the ball of radius $R$ in $\ell^1(L)$, centered in 0, and 
$\ell^1_+(L) := \ell^1(L) \cap \ell^1_+$,
$B_+(L,R) := B(L,R) \cap \ell^1_+$.
%From now on we will always consider $L$ as the identity function, which means $L(k) = k$.

\begin{lemma}
The set $B_+(L,R)$ is relatively compact in the norm topology of $\ell^1$.
\end{lemma}

\begin{proof}
By \emph{Prohorov}'s compactness criterion  a family  of measures  is relatively compact in the weak topology if and only if it is tight.
We have 
$\sum_k k x_k <R$ for any $x\in B_+(L,R)$. Thus for any $n\in \N$ and any $x\in B_+(L,R)$
$$n \sum_{k\geq n} x_k \leq \sum_{k\geq n} k x_k < \sum_k k x_k <R$$
which gives $\sum_{k\geq n} x_k \leq \frac{R}{n}$ for any $n\in \N$ and any $x\in B_+(L,R)$.
So for any $\epsilon>0$ there exists $n \in N$ such that 
$$x (\N \setminus [0,n]) = \sum_{k\geq n} x_k <\epsilon$$
for all $x\in B_+(L,R)$, which means that the tightness condition is satisfied.

%Hence the set $B_+(L,R)$ is relatively compact in the weak topology of $\ell^1$. 
By \emph{Schur}'s theorem, any weakly convergent sequence in $\ell^1$ is actually convergent in the norm of $\ell^1$. 
Therefore the set $B_+(L,R)$ is relatively compact in the topology of $\ell^1$.
\end{proof}

We 
denote by $B_+(L,R]$ the closure of $B_+(L,R)$ in the norm topology of $\ell^1$, which is compact in $\ell^1$, although not in the topology of $\ell^1(L)$.
The set
$$S:=  B_+(L,R]$$ 
 will be the state space for the dynamics considered.

We will assume that the functions defined on $S$ have some regularity. Let $Z$ be a closed convex subset of a normed space $Y$ and  
$f:Z\longrightarrow Y$ a function. Recall that the \emph{directional derivative} $Df(x):Y\longrightarrow Y$ of $f$ in the point $x\in Z$ is a linear form that calculated  in a vector $\xi\in Y$ is defined as $Df(x).\xi := \lim_{t\rightarrow0} \frac{f(x+t\xi) -f(x)}{t}$.
The second order derivative $D^2 f(x)$ is a bilinear form defined as $D^2 f(x). [\xi,\eta] := D(Df(x).\xi).\eta$.
Thus the norms of the derivatives in $Y$ are defined as norms of linear maps:
\begin{align}
||Df(x)||_{Y} &:= \sup_{||\xi|| =1} || Df(x).\xi||_{Y}, \label{der}
\\
||D^2 f(x)||_{Y} &:= \sup_{||\xi||=||\eta||=1} ||D^2 f(x). [\xi,\eta] ||_{Y}.
\end{align}

We say that $f\in\mathcal{C}^1(Z)$ if the function $(x,\xi)\mapsto Df(x).\xi$ is continuous from $Z\times Y$ to $Y$. 
Similarly, $f\in\mathcal{C}^2(Z)$ if the function $(x,\xi,\eta)\mapsto D^2 f(x). [\xi,\eta]$ is continuous from $Z\times Y^2$ to $Y$.
These are subsets of $\mathcal{C}(Z)$ and Banach spaces under the norms 
\begin{align}
||f||_{\mathcal{C}^1(Z)} &:= \sup_{x\in Z} \left\{||Df(x)||_{Y} +||f(x)||_Y \right\}\\
||f||_{\mathcal{C}^2(Z)} &:= \sup_{x\in Z} \left\{||D^2 f(x)||_{Y} +||f(x)||_Y\right\}.
\end{align}

We will use these definitions for the sets $Z=S$, which is convex and compact, and $Y$ to be either $\ell^1$ or $\ell^1(L)$.

\subsection{System of small players}

We describe a so called \emph{fragmentation-coagulation}, or \emph{merging and splitting}, model in which there are $N$  indistinguishable small players that form coalitions to resist to the pressure exerted by a major player, following \cite{V}.

The state space  is
\begin{equation} 
\mathbb{N}^{fin} := \left\{n =(n_1,n_2, \ldots) : \mbox{there is only a finite number of non zero entries}\right\}
\end{equation}
where $n_k\in\mathbb{N}$ denotes the total number of coalitions of size $k$, so the total number of small players is $N=\sum_k k n_k$ and the total number of coalitions is $\sum_k n_k$. The dynamics will be better described in the rescaled space 
\begin{equation}
h\mathbb{N}^{fin} = \left\{x =h n =(x_1,x_2, \ldots)\right\}
\label{hn}
\end{equation}
where $h$ can be taken, for instance, as the inverse of the initial number of coalitions. We want to study the limit as $h\rightarrow 0$. All this $h\mathbb{N}^{fin}$ spaces, as $h$ changes, can be viewed as subspaces of the space $\ell^1$. The total number of players is conserved: this motivates the choice of $L(k)=k$ in the previous section, since  $||x||_{\ell^1(L)} = h N$ 
for any state $x$; we will return to this in 2.5.1.
%Also the limiting system will be assumed to lie in this space.
%, and we want all the states and the limit to stay in an appropriate compact subspace $S$ of this $\ell^1$.

The dynamics evolves in continuous time as a Markov chain. It is described as follows: 
\begin{itemize}
	\item to any randomly chosen pair of coalitions of size $i$ and $j$ is attached a random exponential clock of parameter $h C_{ij}(x,b)$ so that they merge if it rings;
	\item to any randomly chosen coalition of size $i$ is attached an exponential clock of parameter $F_{ij}(x,b)$ such that, if it rings, the coalition splits into two coalitions of size $j$ and $i-j$ .
\end{itemize}
Here the functions $C$ and $F$ may depend on the whole composition $x$  and $b$ is a control parameter which lies in a compact metric space $(E,d)$. 

The minimum of all these exponential random variables is  an exponential random variable with the parameter  
\be
s=s(x,b):= \sum_{i ,j} n_i n_j h C_{ij}(x,b) + \sum_i n_i F_{ij}(x,b). 
\label{sum}
\ee
  %This minimun is obtained by the exponential clock of a certain parameter $\lambda$ with probability $\lambda/s$.
When this minimum clock rings, the 
 system goes from the state $n$ to either $n-e_i - e_j+e_{i+j}$ (two coalitions merge) or $n-e_i + e_j+e_{i-j}$ (a coalition splits). 
 The sequence $(e_i)_{i=1}^\infty$ denotes the standard basis in $\R^\infty$. 
The first case happens if the minimum holds for the clock of parameter $h C_{ij}(x,b)$, thus with probability given by
 $ \frac{h C_{i j}(x,b) n_i n_j}{s(x,b)} $. While the second case happens with probability $\frac{F_{ij}(x,b) n_i}{s(x,b)}$.

%  is the probability that this minimum holds for the clock of parameter $h C_{ij}(x,b)$, so the two coalitions merge, multiplied by the number of coalitions of size $i$ and $j$, i.e. 
%$$ \frac{h C_{i j}(x,b) n_i n_j}{s(x,b)} .$$ 
%Similarly, the probability that the system goes from the state $n$ to $n-e_i + e_j+e_{i-j}$ is the probability that this minimum holds for the clock of parameter $F_{ij}(x,b)$, so the coalition (of sike $i$) brakes, multiplied by the number of coalitions of size $i$, i.e.
%$$\frac{F_{ij}(x,b) n_i}{s(x,b)} .$$

Hence the infinitesimal generator of this Markov chain on the space $\mathbb{N}^{fin}$ is
\begin{align}
\Lambda_{b,n} G (n) &= s(x,b) \sum_{i,j} \frac{h C_{i j}(x,b) n_i n_j}{s(x,b)} \left[G(n- e_i -  e_j+ e_{i+j}) -G(n)\right] \label{genn}
\\
&+ s(x,b)\sum_i \sum_{i<j} \frac{F_{ij}(x,b) n_i}{s(x,b)} \left[G(n- e_i +  e_j+ e_{i-j}) -G(n)\right].\nonumber
\end{align}
%which gives 
%$$\Lambda_{b,n} G(n) = \sum_{i,j} h C_{i j}(x,b) n_i n_j \left[G(n- e_i -  e_j+ e_{i+j}) -G(n)\right] +$$
%\begin{equation}
%+ \sum_i \sum_{j<i} F_{ij}(x,b) n_i \left[G(n- e_i +  e_j+ e_{i-j}) -G(n)\right] .
%\label{genn}
%\end{equation}
Equation (\ref{genn}) can be equivalently presented as the infinitesimal generator
\begin{align}
\Lambda_{b,h} G(x) &= \frac1h \sum_{i,j} C_{i j}(x,b) x_i x_j \left[G(x-h e_i - h e_j+h e_{i+j}) -G(x)\right] \label{gen}\\
&+ \frac{1}{h}\sum_i \sum_{j<i} F_{ij}(x,b) x_i \left[G(x-h e_i + h e_j+h e_{i-j}) -G(x)\right] \nonumber
\end{align}
of the Markov chain describing the system of small players on the space $h\mathbb{N}^{fin} \subset \ell^1(\mathbb{N})$, for every 
$G\in\mathcal{C}(S)$.
%, where $S$ is a suitable compact subset of $\ell^1(\mathbb{N})$, to be determined later.

\begin{notation}
$X^h(t,x,b)$ is  the state (in $S$) at time $t$ of the Markov Chain given by this generator (\ref{gen}) which is in $x$ at $t=0$, under the control parameter $b$ given by the major player.
\label{def1}
\end{notation}

 The process $X^h(t,x,b)$ describes  the evolution of the coalitions of small players, which will be also called the system with $N$ agents.

\subsection{Limiting system}

The limiting deterministic evolution, the \emph{mean field limit}, is described by  the so called \emph{Smoluchovski equation}. For every $x$ in the compact subset $ S\subset\ell^1$ the ODE  for the component $i$ is
\begin{align}
\dot{x}_i = f_i(x,b) &:=\sum_{j<i} C_{j,i-j} (x,b) x_j x_{i-j} -2\sum_j C_{ij} (x,b) x_i x_j \label{effe}\\
&+ 2 \sum_{j>i} F_{ji}(x,b)x_j -\sum_{j<i}F_{ij}(x,b)x_i \nonumber.
\end{align}

\begin{notation}
$X(t,x,b)$ is the \emph{flow} at time $t$ of the ODE 
\be
\dot{x}= f(x,b) 
\label{f}
\ee
starting in $x$ at $t=0$ under the control parameter $b\in E$, where $f$ is given by (\ref{effe}). In integral form
\be
X(t,x,b) = x + \int_0^t f(X(s,x,b),b)ds.
\ee
\end{notation}

We view the dynamics given by a deterministic ODE as a Markov process. The semigroup is
\be
U_t G(x) = G(X(t,x))
\label{sem}
\ee
for every 
$G\in\mathcal{C}(S)$, and its generator is  given by
\be
\Lambda G(x) := \sum_i f_i(x) \frac{\partial G}{\partial x_i}(x),
\label{infg}
\ee
for any $G\in\mathcal{C}^1(S)$. The first order partial differential operator defined in 
(\ref{infg}) has characteristics which solve equation (\ref{f}).

So, for the limiting ODE given by (\ref{effe}), the corresponding infinitesimal generator given by (\ref{infg}) is 
\begin{align}
\Lambda_{b} G(x) &=  \sum_{i,j} C_{i j}(x,b) x_i x_j \left[\frac{\partial G}{\partial x_{i+j}} - \frac{\partial G}{\partial x_i} - \frac{\partial G}{\partial x_j}\right] 
\label{inff}\\
&+ \sum_i \sum_{j<i} F_{ij}(x,b) x_i \left[\frac{\partial G}{\partial x_{i-j}} + \frac{\partial G}{\partial x_j} - \frac{\partial G}{\partial x_i}\right]
\nonumber
\end{align}
for any $G\in\mathcal{C}^1(S)$ and $b\in E$.

We can thus deduce that pointwise convergence of the generators holds. Namely, for the generators of the Markov chains defined by (\ref{gen}) and the generator of the deterministic limit defined by (\ref{inff}), we obtain
\be
\lim_{h\rightarrow0} \Lambda_{b,h}G(x) = \Lambda_b G(x)
\ee
for every $G\in\mathcal{C}^1(S)$, $x\in S$ and every $b\in E$.

%\subsubsection{Convergence of the dynamics}

We show  moreover the convergence in law, for any fixed parameter $b\in E$, of the processes $X^h$ to $X$ in the Skorokhod space $D([0,T],S)$ of cadlag functions, which is the right space where to study these processes. The convergence is then also in probability, as the limit is deterministic, and hence a constant in the Skorokhod space.

\begin{proposition}
Let all the functions $C_{ij}$ and $F_{ij}$ be in $\mathcal{C}^1(S)$. Suppose that the initial points $x(h)$ converge in $\ell^1$ to $x_0$, as $h\rightarrow0$. Then the processes 
$X^h(\cdot, x(h), b)$ converges in law on the Skorokhod space $D([0,T],S)$ to $X(\cdot, x_0, b)$, as $h\rightarrow0$, for any $b\in E$; whereas the processes are defined in notations 1 and 2.
\end{proposition}

\begin{proof}
Let $b\in E$ be fixed. The set $\mathcal{C}^1(S)$ is a dense linear subspace of $\mathcal{C}(S)$ and under the assumption of smooth $C_{ij}$ and $F_{ij}$ it is  invariant under the limiting semigroup $(U_t)$ defined in (\ref{sem}), since the function $f$ defined in (\ref{effe}) turns out to be in $\mathcal{C}^1(S)$. So by (\cite{G}, proposition 17.9)
%(\ref{inv})
 the set $\mathcal{C}^1(S)$ is a core for the generator $\Lambda_b$ defined in (\ref{inff}). 

Then expanding $G$ in Taylor series we have that 
\be
\lim_{h\rightarrow0} \Lambda_{b,h}G = \Lambda_b G
\ee
uniformly for every $G\in\mathcal{C}^1(S)$. Thus the claim follows from (\cite{G}, theorem 17.25) which characterizes the convergence of processes in $D([0,T],S)$.
% as the domain of $\Lambda_h$ is the whole $\mathcal{C}(S)$.
%(\ref{eqconv})
\end{proof}

\subsection{Controlled systems}

Here we deal with the system of small players under some control, i.e. a strategy given by the major player. We assume that this major player focuses in finite horizon time $n\tau$ and can update her strategy only in discrete times
$$k\tau \qquad k=0,1,\ldots, n-1 .$$
$\tau>0$ and $n\in \mathbb{N}$ are fixed, and in each time step the controller can change the control parameter $b$ regarding what has happened in the time interval. 

The starting point of the Markov chain is $x_0=x(h)\in h\mathbb{Z}^{fin}$, with the control parameter $b_0$. After the first time step the Markov chain is in the state $x_1=X_h(\tau,x_0,b_0)$. Now the major player  can change the parameter, so it becomes $b_1=b_1(x_1)$ that may depend on the current state of the system. She repeats the same procedure at each time step and therefore, in the end, we get what is called a \emph{policy}.

\begin{notation}

A \emph{policy} is a sequence of decision rules 
\be
\pi= (\pi_0,\pi_1,\ldots,\pi_{n-1})
\label{pol}
\ee
that specify the action of the mayor player at each time step, with
\be
\pi_k=\pi_k(x_k)
\label{bk}
\ee
and
\be
x_k=X^h(\tau,x_{k-1},\pi_{k-1}).
\label{xk}
\ee

Let  $X_\pi^h(t,x_0)$ denote the state of the system at time $t$ when the controller applies policy $\pi$, starting from the initial point $x_0$. To shorten the notation we shall sometimes write  $X_\pi^h(t)$ instead of $X_\pi^h(t,x_0)$. It is called the controlled system of small players.
\label{def3}
\end{notation}

 Equation (\ref{xk}) can be also written as $x_k= X_\pi^h(k\tau)$.
At each time step $k\tau$ the controller has an \emph{instantaneous reward} 
$B(x_k, b_k)$ and in the end she has a \emph{final reward} $V_0(x_n)$.
Our goal is to find a strategy that maximizes 
\begin{align}
V_{\pi,n}^h(x(h)) :&=E[\tau B(x_0,\pi_0)+\ldots+\tau B(x_{n-1},\pi_{n-1}) +V_0 (x_n) ]
\label{Vpi}\\
&=E\left[\left.\sum_{k=0}^{n-1} \tau B(X_{\pi}^h(k\tau), \pi(X_{\pi}^h(k\tau))+V_0(X_{\pi}^h(n\tau)) \right|X_{\pi}^h(0)=x_0\right]
\nonumber
\end{align}
where $B$ and $V_0$ are given continuous functions. It is called the \emph{value} for the system with $N$ players. The maximum over all possible policies is then the \emph{optimal value} for the system with $N$ agents
\be
V_n^h(x(h)):= \sup_{\pi}V_{\pi,n}^h(x(h)) .
\label{V}
\ee

We may want to find this optimum value via the usual \emph{dynamic programming} method. First of all we define the \emph{Shapley operator}
\be
S[h] V(x) := \sup_{b\in E} [\tau B(x,b) + E(V(X_h(\tau,x,b)))]
\ee
and then by backward recurrence 
\be
V_k^h = S[h]V_{k-1},
\ee
hence we get
\be
V_n^h =S[h]^n V_0.
\label{bell}
\ee

However this procedure might be unfeasible to calculate practically when the number of players increases. So we will consider the optimum of the limit and then study how close these optima are.

\subsubsection{Controlled limiting system}

We want to study the mean field limit system, given by equation (\ref{effe}), in a classical control theory setup.
Recall that $(E,d)$ is a compact metric space, the one where the parameter $b$ lies. 

\begin{notation}
We define an \emph{action function} to be a piecewise Lipschitz function from finite horizon time to  E
$$\alpha:[0,T]\rightarrow E .$$
\label{action}
\end{notation}

 We note that an action function is different from a policy, because the latter depends on the state of the system at each step, while the former does not.  Thus in this context we rewrite equation (\ref{f}) where $f$ is defined in (\ref{effe}) as
\be
\dot{x} = f(x,\alpha), 
\label{fal}
\ee
meaning $\dot{x}(t)= f(x(t),\alpha(t))$ for every $t\geq 0$, considering hence $b=\alpha(t)$, i.e. the control parameter is a function of the time.

\begin{notation}

$X(t,x,\alpha)$ is the \emph{flux} of the ODE (\ref{fal}), i.e. the solution at time $t$ that is in $x$ at $t=0$ under the control parameter $b=\alpha=\alpha(s)$. In integral form
\be
X(t,x_0,\alpha)= x_0+ \int_0^t f(X(s,x_0, \alpha), \alpha(s))ds .
\label{fint}
\ee
\label{def4}
\end{notation}

We are in a finite horizon time $T$ and now we want to maximize
\be
v_{\alpha}(x):= \int_0^T B(X(s,x,\alpha),\alpha(s))ds +V_0(X(T,x,\alpha))
\label{val}
\ee
where $B$  and $V_0$ are the same as in (\ref{Vpi}). This is the \emph{value} of the limiting system.
The \emph{optimal value} is then
\be
v(x) =\sup_{\alpha}v_{\alpha}(x) .
\label{v}
\ee

Our aim is to study how and under what assumptions we have the convergence of the optimum  of the system of small players (\ref{V}) to this optimum (\ref{v}). In fact we want both $h$ and $\tau$ tend to $0$. To achieve this goal we need further auxiliary systems, to get also the convergence for every policy and every action function.

%\section{Assumptions in the model}

%This section is our main contribution. Here we want to show how the assumptions stated in section 1.4 hold in our fragmentation-coagulation model. 
%Of course we have to assume some regularity on the functions involved in the model. 
%The conditions needed will be summarized in the end of the section. 
%First of all we need to define an appropriate compact subset $S$ of $\ell^1=\ell^1(\mathbb{N})$ where both the dynamic of the system of small players and of the mean field limit lie. 

%\subsubsection{Remark}

%Observe that in $\ell^1$ the weak convergence in the sense of probability and the weak convergence in  the sense of functional analysis are exactly the same. Indeed a sequence $(x^N)$ in $\ell^1$ converges weakly to $x$, in the sense of probability, if $\lim_{N\rightarrow\infty} (x^N,a) = (x,a)$ for any function $a\in \mathcal{C}_b(\N)$, i.e.
%\be
%\lim_{N\rightarrow\infty} \sum_k a_k x^N_k = \sum_k a_k x^N_k.
%\label{weak}
%\ee
%While a sequence $(x^N)$ in $\ell^1$ converges weakly to $x$, in the sense of functional analysis, if $\lim_{N\rightarrow\infty} (x^N,a) = (x,a)$ for any sequence $a\in \ell^{\infty}$, which is the topological dual of $\l$, i.e. the same equation (\ref{weak}) holds.
%Thus the two convergences are equal, since $\mathcal{C}_b(\N)=\ell^{\infty}$.

\subsection{Stability of $S$}

We  show that the state space $S:= B_+(L,R]$ is stable for all the dynamics considered.
We need some regularity for the functions involved in the model.
%$$C_{ij}(\cdot,b), F_{ij}(\cdot,b): S \longrightarrow \R^+ .$$
We require  that all the functions $C_{ij}(x,b)$ and $F_{ij}(x,b)$ are positive and in $\mathcal{C}^2(S)$ in the variable $x$, for any $b$, i.e. twice continuously differentiable on the compact subspace $S \subset \ell^1$, in the topology of $\ell^1$.
Since $S$ is convex we can take the directional derivatives in every direction, so  we have
\begin{eqnarray}
&C:=\sup_{i,j} C_{ij}(x,b)<\infty, \quad &F=\sup_i \sum_{j<i}F_{ij}(x,b)<\infty
\label{CF}
\\
&C(1) :=\sup_{i,j,k} \left|\frac{\partial C_{i,j}}{\partial x_k}(x,b)\right| <\infty,
\quad &F(1) :=\sup_{i,k} \sum_{j<i}\left|\frac{\partial F_{i,j}}{\partial x_k}(x,b)\right| <\infty
\label{CF1}
\\
&C(2):= \sup_{i,j,k,l}\left|\frac{\partial^2 C_{i,j}}{\partial x_k \partial x_l}(x,b)\right|<\infty,
\quad &F(2) := \sup_{i,k,l} \sum_{j<i} \left|\frac{\partial^2 F_{i,j}}{\partial x_k \partial x_l}(x,b)\right|<\infty.
\label{CF2}
\end{eqnarray}
These constants are all finite as $S$ is compact.

%We recall that for a function $f:Y\longrightarrow Y$ between normed spaces   its \emph{ derivative} $Df(x):Y\longrightarrow Y$ in the point $x$ is a linear form that calculated  in a vector $\xi\in Y$ is defined as $Df(x).\xi := \lim_{t\rightarrow0} \frac{f(x+t\xi) -f(x)}{t}$
%and the second order derivative $D^2 f(x)$ is a bilinear form defined as $D^2 f(x). [\xi,\eta] := D(Df(x).\xi).\eta$
%Thus the norms of the derivatives in $Y$ are defined as norms of linear maps:
%\be
%||Df(x)||_{Y} := \sup_{||\xi|| =1} || Df(x).\xi||_{Y},
%\qquad ||D^2 f(x)||_{Y}:= \sup_{||\xi||=||\eta||=1} ||D^2 f(x). [\xi,\eta] ||_{Y}
%\label{der}
%\ee

Let us recall that $L$ is the identity, i.e. $L(k)=k$. Therefore, using the above equalities in (\ref{effe}) we get that also $f:S\rightarrow \ell^1$ is twice continuously differentiable 
(in $\mathcal{C}^2(S)$) as a map both in $\ell^1$ and in $\ell^1(L)$ with the following bounds
\begin{eqnarray}
||f(x)||_{\ell^1}&\leq& 3C ||x||_{\ell^1}^2 +3 F||x||_{\ell^1}
\label{uno}
\\
||f(x)||_{\ell^1(L)} &\leq& 3(C ||x||_{\ell^1} +3F) ||x||_{\ell^1(L)}
\\
\left\|Df(x)\right\|_{\ell^1} &\leq& 6 C ||x||_{\ell^1} +3 F  +3 [C(1) ||x||_{\ell^1} +F(1)]||x||_{\ell^1}
\label{due}
\\
\left\|Df(x)\right\|_{\ell^1(L)} &\leq& 8C ||x||_{\ell^1(L)} +3F +3 [2C(1) ||x||_{\ell^1} +F(1)]||x||_{\ell^1(L)}
\label{l1L}
\\
||D^2 f(x)||_{\ell^1}&\leq& 6 [C +F(1) +[C(1) +F(2)]||x||_{\ell^1} +C(2) ||x||_{\ell^1}^2]
\label{tre}
\\
||D^2 f(x)||_{\ell^1(L)}&\leq& 9 [C +F(1) +[C(1) +F(2)]||x||_{\ell^1(L)} +C(2) ||x||_{\ell^1(L)}^2].
\end{eqnarray}

We recall now some fact about ODEs in Banach space of measures.
%\subsubsection{Dynamics in $\ell^1$}
 In the Markovian dynamics of the system of small players every state represents the number of coalitions of different sizes, which is of course positive. Hence we are interested in an evolution $f:S\longrightarrow\ell^1$ for the dynamic (\ref{f}) $\dot{x} =f(x)$
that preserves positivity, i.e. such that for any initial point $x\in \ell^1_+$ the solution $X(t,x)$ belongs to $\ell^1_+$ for any $t\geq0$. We say that $f$ must be \emph{conditionally positive}, in the following sense:

\begin{definition}
A function $f:\ell^1\longrightarrow\ell^1$ is said to be \emph{conditionally positive} if 
%for any $x\in \l_+$ the negative part of $f(x)$ is absolutely continuous with respect to $x$, meaning that 
for any 
$x\in \ell^1_+$ with $x_k=0$ one has $f_k(x)\geq0$.
\end{definition}

 Further, we need the following definitions.

\begin{definition}
A function $f:\ell^1_+\longrightarrow\ell^1$ is called \emph{$L$-subcritical} if  
\be
\sum_k L(k) f_k(x) \leq 0.
\ee
\end{definition}
As a motivation, we observe that $\frac{d}{dt} ||x||_{\ell^1(L)} \leq 0$ if $f$ is $L$-subcritical and $\dot{x}=f(x)$. 

\begin{definition}
A function $f:\ell^1_+\longrightarrow\ell^1$ is said to satisfy the \emph{Lyapunov condition}  if 
\be
\sum_k L(k) f_k(x) \leq a \sum_k L(k) x_k +b
\ee
for some constant $a$ and $b$, for all $x\in\ell^1_+$.
\end{definition}

The main result concerning the dynamics in $\ell^1$ is the following lemma. 
%It can be extended, when needed, to the dynamics in more general Banach spaces of measures.

\begin{lemma}
Assume that the function $f$ is conditionally positive, satisfies the Lyapunov condition
 and is Lypschitz continuous in the norm of $\ell^1(L)$ on any bounded set of $\ell^1(L)_+$. Then for any $x\in\ell^1(L)_+$ the Cauchy problem (\ref{f}) has a unique global (defined for all times) solution $X(t,x)$ in $\ell^1_+(L)$. Moreover
\be
X(t,x) \in B_+(L, e^{at}(||x||_{\ell^1(L)} +bt)).
\ee
In particular if $f$ is $L$-subcritical then any  $B_+(L,R)$ is invariant.
\label{lem}
\end{lemma}

\begin{proof}

By local Lipschitz continuity and conditional positivity, evolution (\ref{f}) is locally well-posed and preserves positivity. Moreover, by the Lyapunov condition,
$$(L,X(t,x)) \leq (L,x) + \int_0^t [a(L, X(s,x)) +b]ds $$
where $(L,x) = \sum_k L(k) x_k$ is the duality between functions and measures.
So by Gronwall's lemma and the preservation of positivity
$$0\leq (L,X(t,x)) \leq e^{at} [(L,x)+ bt] ,$$
implying that the solution can be extended to all times with required bounds.

\end{proof}

We have found  that, if the assumptions of the lemma are satisfied,   the set  $B_+(L,x_0]$  is invariant under an $L$-subcritical evolution $f$ and the ODE has a unique global solution, starting from $x_0$.
% So we call
%$$S= B_+(L,R]$$
%and we would like this set to be invariant also for the system of small players.
%\subsubsection{Hypothesis of  lemma \ref{lem}}

 Let us check that the assumptions of  lemma \ref{lem} are satisfied for $f$ defined in (\ref{effe}).
% and the identity as the Lyapunov function, i.e. $L(k)=k$. 
The function $f$  is conditionally positive because its domain is $S$, which is a subset of $\ell^1_+$,  and the functions $C_{ij}(x,b)$ and $F_{ij}(x,b)$ are positive. 
Further, if we consider the function $G\in \mathcal{C}^1 (S)$ defined by  $G(x) = ||x||_{\ell^1(L)} = \sum_k k x_k $ and apply the generator (\ref{inff})  
%$$\Lambda G(x)= \sum_i f_i(x) \frac{\partial G}{\partial x_i}(x)$$
%$$=  \sum_{i,j} C_{i j}(x,b) x_i x_j \left[\frac{\partial G}{\partial x_{i+j}} - \frac{\partial G}{\partial x_i} - \frac{\partial G}{\partial x_j}\right] 
%+ \sum_i \sum_{j<i} F_{ij}(x,b) x_i \left[\frac{\partial G}{\partial x_{i-j}} + \frac{\partial G}{\partial x_j} - \frac{\partial G}{\partial x_i}\right]
%$$
to this function then we have $\sum_k k f_k(x) = 0$,
since the derivatives of $G$ are $\frac{\partial G}{\partial x_k}(x) = k$. This implies that $f$ is $L$-subcritical.

Considering  equation (\ref{l1L}) and thanks to the boundedness of $S$ in $\ell^1(L)$  
%we obtain that the function $f$ is continuously differentiable in $\l(L)$ with a bounded derivative, since we are in the space $S$ compact in $\l(L)$. 
we obtain that $f$ is Lipschitz continuous as a map in $\ell^1(L)$. So all the assumptions of lemma \ref{lem} are satisfied, showing the well posedness of the problem and the invariance of $S$.
We would like this set to be invariant also for the system of small players.

\subsubsection{State space for the small players}

For any $h$, $X^h(t, x)$ is a  continuous-time Markov chains  on $h \N^{fin} \cap \ell^1_+$.
 Let us say
that  $X^h(t, x)$ are \emph{L-non-increasing}, if any jump of $X^h(t, x)$ cannot increase $L$. In this case a trajectory $X^h(t, x)$ stays forever in $B_+(L,R)$ whenever the initial point
$x\in B_+(L,R)$. Moreover $X^h$ is \emph{L-subcritical} in the sense that its generator $\Lambda_{h,b}$
satisfies the inequality $\Lambda_{h,b}(L)\leq0$.

The Markov chains $X^h(t, x(h))$ are  $L$-non increasing and have bounded generators. In fact the state space of the coalitions of small players is actually finite for any fixed $h$. Indeed we recall that if $X^h(t, x)$ is in the state $x$ then $x_k$ is $h$ times the number of coalitions of size $k$. The total number of small players is fixed $N=N(h)$ for any $h$, so $x_k=0$ for any $k\geq N$ and $x_k\leq N/h$ for any $k\leq N$, meaning that the state space is finite.

The total number of small players $N$ is of course constant. So the norm in $\ell^1(L)$ of the states $x$ of the Markov chain $X^h(t)$ is conserved:
\be
||x||_{\ell^1(L)} =  \sum_{k=1}^N k x_k = h\sum_{k=1}^N k n_k = h N(h).
\label{xNh}
\ee
Hence if the initial point $x(h)$ is in $S = B_+(L,R]$ then any state is in $S$, i.e. the set $B_+(L,R]$ is invariant for the dynamics of $X^h$.

\begin{notation}
$S(h)$ is the finite state space of the Markov chain $X^h$, the system of $N=N(h)$ small players. It is a subset of the compact $B_+(L,R]$ in $\ell^1$ and a subset of $\R^N$ and  of the  set $h \mathbb{N}^{fin}$.
%$h \mathcal{N}^N$, whereas $\mathcal{N} = \{1,2, \ldots, N\}$. 
Denote by $M(h)$ the number of elements of $S(h)$
$$S(h):=  h \mathbb{N}^{fin} \cap B_+(L,R] .$$
\label{S(h)}
\end{notation}

 So we can  define $S:= B_+(L,R]$
for a suitable $R$ such that this set contains all the initial data $x(h)$ and $x_0$. Such an $R$  exists because we will consider $\h x(h) = x_0$ and then the sequence is bounded. $S$ is  the compact set in $\ell^1$ invariant for all the dynamics considered.
Thanks to (\ref{xNh}) we have  
\be
N(h) \leq \frac{R}{h}
\label{NhR}
\ee
for any $h$.
Further
$\h M(h) = \h N(h) = +\infty $
and the finite spaces $S(h)$ are decreasing, i.e.  if $h>l$ then
$S(h)\subset S(l),$
and
$ \bigcup_{h>0} S(h) \subseteq S .$
Moreover
\be
S(h)  \subset S \subset \ell^1(L) \subset\ell^1 \subset \ell^2.
\label{inclusion}
\ee

%Now we want to see that all these assumptions are satisfied in our model.

%\subsection{Conditions on $C_{ij}$ and $F_{ij}$}

%Consider then the function 
%$$f:S\times E\longrightarrow\l$$
% for the limiting system described in (\ref{effe}):
%$$
%\dot{x}_i = f_i(x) :=\sum_{j<i} C_{j,i-j} (x,b) x_j x_{i-j} -2\sum_j C_{ij} (x,b) x_i x_j +
%$$
%$$
%+ 2 \sum_{j>i} F_{ji}(x,b)x_j -\sum_{j<i}F_{ij}(x,b)x_i .
%$$

\subsection{Assumptions in the model}

Let us summarize the assumptions we make on our model. Recall that in the system of $N$ small players the controller acts at time steps $k\tau$ for $k=0,1,\ldots,n-1$. 
\begin{itemize}
	\item \textbf{(H1)} $S= B_+(L,R]$ is the  state space of all the dynamics considered, and the initial states lie in $S$;
	\item \textbf{(H2)} The functions $C_{ij}(x,b)$ and $F_{ij}(x,b)$ are positive and in $\mathcal{C}^2(S)$ in the variable $x$, for any $b\in E$, and Lipschitz continuous in the variable $b$, for any $x\in S$;
	\item \textbf{(H3)} The rewards $B(x,b)$ and $V_0(x)$ are Lipschitz continuous in $x$ in the $\ell^2$-norm, uniformly in $b$, and $B$ is bounded;
	\item \textbf{(H4)} The time step $\tau= \tau(h)$ depends on $h$, as well as  $N=N(h)$ does, and
	$$\h \tau(h)= \h \tau(h) \sqrt{N(h)} = 0 ;$$
	\item \textbf{(H5)} The horizon $T$ is fixed and the number of steps is  
	\be
	n(h):= \left\lfloor \frac{T}{\tau(h)}\right\rfloor
	\label{T}
	\ee
	for any $h$, which tends to infinity, as $h$ tends to 0;
	\item \textbf{(H6)} The rescaling parameter $h$ of the model is such that
\be
\h h (N(h))^2 =\h \tau(h) (N(h))^2=0.
\label{assump}
\ee
\end{itemize}
Equivalently, we can think of studying the limit as $N$ tends to infinity. So the parameter $h=h(N)$ has to satisfy the latter conditions and it represents the rescaling parameter for the system of $N$ players.

%We recall again that the intensity function of the model is $I(h)=\tau(h)$,  and that the other functions involved are
%$$I_0(h) := \sqrt{N(h)} \frac{\tau}{2}(R_1 + h R_2)$$
%$$I_1(h) := \tau(CR^2 +F) \qquad I_2(h) := (CR^2 +F) [ \tau(CR^2 +F) + h].$$
%Moreover we must finally assume that 
%This stronger assumption is needed in order to obtain the mean field convergence, as we see in 4.3.2.

\section{Mean field convergence}

In this section we present our main results. We follow the ideas in \cite{M}, hence we firstly introduce the two auxiliary systems.

\subsubsection{First auxiliary system}

This is a system with $N$ agents controlled by an action function borrowed  from the mean field limit. More precisely, let $\alpha$ be an action function that specifies the action to be taken at time $t$. Although $\alpha$ has been defined for the limiting system, it can also be used for the system with $N$ players. In this case, the action function $\alpha$ can be seen as a policy that does not depend on the state of the system.

At step $k$, the controller applies action 
$$\alpha_k:=\alpha(k  \tau),$$
so (\ref{bk}) gives a policy $(\alpha_0,\ldots,\alpha_{n-1})$ as in (\ref{pol}), but independent of the state of the system.

By abuse of notation, we denote by $X_\alpha^h(t)$ as in (\ref{def3}) the state of the system at time $t$ when applying the policy derived from the action function $\alpha$ as explained above. In what follows policies will always be denoted by $\pi$ and action functions by $\alpha$. Here (\ref{xk}) becomes
$$x_k=X^h(\tau, x_{k-1}, \alpha_{k-1})= X_\alpha^h(k \tau),$$ starting from initial point $x_0$ with control parameter $\alpha_0 = \alpha(0)$.

The value for this system, similarly to (\ref{Vpi}), is defined by
\begin{align}
V_{\alpha, n}^h (x_0):&= E[\tau B(x_0, \alpha_0) +\ldots +\tau B(x_{n-1},\alpha_{n-1}) + V_0(x_n)] 
\\
&=
E\left[\left.\sum_{k=0}^{n-1} \tau B(X_{\alpha}^h(k \tau), \alpha(k \tau)) + V_0(X_\alpha^h(n\tau)) \right| X_\alpha^h(0) = x_0\right].
\nonumber
\end{align}

\subsubsection{Second auxiliary system}

The method of proof uses a second auxiliary system in which trajectories are considered. This is a limiting system controlled by an action function derived from the policy of the original system with $N$ agents.

Consider the system with $N$ players under policy $\pi$. The stochastic process $X_\pi^h=X_{\pi,n(h)}^h$ is  defined on some probability space $\Omega$. 
To every $\omega\in\Omega$ there corresponds a trajectory $X_\pi^h(\omega)$, and for every $\omega\in\Omega$ we define  the piecewise constant action function $A_\pi^h(\omega)$, as explained in the following

\begin{notation}
$$A_\pi^h(\omega): [0,T]\rightarrow E$$
is an action function such that
\begin{itemize}
	\item this random function is constant on each interval $[k \tau, (k+1) \tau[$ for any $k=0,1,\ldots,n-1$;
	\item $A_\pi^h(\omega)(k \tau):=\pi_k = \pi_k(X_\pi^h(k\tau))$ is the action taken by the major player of the system with $N$ agents at time slot $k\tau$, under policy $\pi$.
\end{itemize}
\label{def9}
\end{notation}

Recall from notation \ref{def4} that for any $x_0\in S\subset\ell^1(\mathbb{N})$ and any action function $\alpha$, $X(x_0,\alpha)$ is the solution of the ODE (\ref{fal}). For every $\omega$, $X(t, x_0, A_\pi^h(\omega))$ is the solution of the limiting system with action function $A_\pi^h(\omega)$, as in (\ref{fint}), i.e.
\be
X(t,x_0,A_\pi^h(\omega))= x_0+ \int_0^t f(X(s,x_0, A_\pi^h(\omega)), A_\pi^h(\omega)(s))ds .
\label{fluxalphapi}
\ee
The value function for this system is as in (\ref{val}).

 When $\omega$ is fixed, $X(t,x_0,A_\pi^h(\omega))$ is a continuous time deterministic process corresponding to one trajectory $X_\pi^h(\omega)$. When considering all possible realizations of $X_\pi^h$, $X(t,x_0,A_\pi^h)$ is a random, continuous time function \emph{coupled} to $X_\pi^h$, i.e. a stochastic process. Its randomness comes only from the action term $A_\pi^h$, in the ODE (\ref{fal}). In the following  we omit the dependence on $\omega$ in our writing. $A_\pi^h$ and $X_\pi^h$ will always designate the processes corresponding to the same $\omega$.

%In this chapter we state the results presented in \cite{M} and write their proofs, for the sake of completeness. We add some comment and adjust some proof to our context. Moreover we show  the proof of lemma (\ref{lem1}), which was not present in the cited article. In the last section we see the algorithm of the authors for computing an optimal policy starting from an optimal action for the mean field limit and we add some remark about the solutions of the Hamilton-Jacobi-Bellman equation.

\subsection{Main results}

The main result establishes the convergence of the optimization problem for the system with $N$ players to the optimization problem for the mean field limit, through their value functions. 

\begin{theorem}
Under assumptions (H1)-(H6), if $\lim_{h\rightarrow0} x(h) =x_0$ almost surely, respectively in probability, then
\be
\h V^h(x(h))= v(x_0)
\label{limv}
\ee
almost surely, respectively in probability, where $V^h$ and $v$ are the optimal values defined in (\ref{V}) and (\ref{v}).
\label{teo2}
\end{theorem}

The second result states that an optimal action function for the mean field limit provides an asymptotically optimal strategy for the system with $N$ agents.
Let us denote by
  $(\hat{X}^h_\alpha(t))_{t\geq0}$ the continuous time process which is the  affine interpolation of $X^h_\alpha(t)$ (the first auxiliary system) in the points $k\tau$ and 
	%i.e. 
%\begin{itemize}
	%\item $\hat{X}^h(t)$ is affine on the intervals $[k \tau, (k+1) \tau]$ for $k=0,\ldots, n-1$ 
	%\item
%	$\hat{X}^h(k \tau)) = X^h(k\tau)$.
%\end{itemize}
similarly by $\hat{X}^h_\pi(t)$  the  affine interpolation of $X^h_\pi(t)$ under policy $\pi$. 
%\label{def8}
%\end{notation}

\begin{theorem}
Under assumptions (H1)-(H6), let $\alpha$ be a piecewise Lipschitz continuous action function on $[0,T]$, of Lipschitz constant $K_\alpha$, and with at most $p$ discontinuity points. 
%Let $\hat{X}^h_\alpha(t)$ be the linear interpolation, as in notation \ref{def8}, of the first auxiliary system $X^h_\alpha(t)$, with random starting point $X^h(0)$. 
Then there exist functions $J$, $I_0'$ and $B'$ satisfying 
$$
\h I_0'(h,\alpha) = \h J(h,T) =0,\quad \lim_{\substack{h\rightarrow0 \\ \delta\rightarrow0}}B'(h,\delta)=0
$$
such that
for all $\epsilon>0$ 
\be
P\left\{\sup_{0<t<T}\left\|\hat{X}^h_\alpha(t)-X(t,x_0, \alpha)\right\|
>\left[\left\|X^h(0)-x_0\right\|+I_0'(h,\alpha)T +\epsilon\right]e^{L_1 T}\right\}
\leq\frac{J(h,T)}{\epsilon^2}
\label{T3}
\ee
and
\be
\left|V_\alpha^h(X^h(0))- v_\alpha(x_0)\right|\leq B'\left(h,\left\|X^h(0)-x_0\right\|\right).
\label{68}
\ee

\label{teo3}
\end{theorem}

%In the above equations  $J$, $I_0'$ and $B'$ satisfy 
%$$
%\h I_0'(h,\alpha) = \h J(h,T) =0,\quad \lim_{\substack{h\rightarrow0 \\ \delta\rightarrow0}}B'(h,\delta)=0.
%$$

%Equation (\ref{T3}) gives the convergence in probability of the process $\hat{X}^h_\alpha$ to the process $X(\cdot, x_0,\alpha)$ for any action function $\alpha$, if $\h X^h(0) = x_0$
%almost surely. This means the convergence in the space of trajectories   
%$$D([0,T],S)=\{g:[0,T]\longrightarrow S \mbox{     cadlag} \} ,$$
%which is the right space for Markov chains and also for continuous processes like the linear interpolations $\widehat{X}^h$ considered.
% which is usually equipped with the Skorokhod metric. However, since the time horizon is finite and the limit process $X$ is continuous, a sequence converges to $X$ in the Skorokhod topology if and only if it converges in the uniform norm.
% $$||g||_{\infty} = \max_{t\in [0,T]} |f(t)|.$$
%Thus
%theorem \ref{teo3} shows the convergence , for every $\alpha$, of $\hat{X}^h_\alpha$ to $X(\cdot, x_0,\alpha)$ in probability 
%in the uniform norm, which is  then equivalent to the convergence in probability 
%in  the Skorokhod topology;  this is equivalent, as the limiting process is deterministic and so a constant in $D([0,T],S)$, to the weak convergence in the Skorokhod topology. 

%The weak convergence of processes $X^h$ to $X$ can be proved also in an other way, as it turns to be equivalent to the convergence of their generators in a core of the limiting semigroup (see \cite{G}). 

Inequality (\ref{68}) implies also that if 
\be
\h X^h(0) = x_0
\label{xh}
\ee
 almost surely, respectively in probability, then
\be
\h V_\alpha^h(X^h(0)) = v_\alpha (x_0)
\label{valpha}
\ee
almost surely, respectively in probability.

The following corollary combines Theorems \ref{teo2} and \ref{teo3}. It states that an optimal action function for the limiting system is asymptotically optimal for the system of small players.

\begin{corollary}
If $\alpha_\ast$ is an optimal action function for the limiting system and if $ \h X^h(0) = x_0$ almost surely, respectively in probability, then
\be
\h \left|V_{\alpha_\ast}^h(X^h(0)) - V^h(X^h(0))\right|=0
\label{Vhalpha}
\ee
almost surely, respectively in probability.
\label{cor4}
\end{corollary}

\begin{proof}
The assumption says that there exists an action function $\alpha_\ast$ that maximizes $v$, i.e.
$v(x_0)=v_{\alpha_\ast}(x_0) =\max_{\alpha} v_\alpha(x_0) .$
Hence we have
$$\left|V_{\alpha_\ast}^h(X^h(0)) - V^h(X^h(0))\right| \leq \left|V_{\alpha_\ast}^h(X^h(0))- v_{\alpha^\ast}(x_0)\right| + \left|V^h(X^h(0))- v(x_0)\right|.$$
So, if (\ref{xh}) holds,  the first modulus goes to 0 by (\ref{valpha}) and the second by (\ref{limv}). Therefore (\ref{Vhalpha}) is given combining Theorems 1 and 2.
\end{proof}

%As the reward function $B(x,b)$ is bounded and the time horizon is finite, the set of values 
%when starting from initial condition $x$ 
%$$\{v_\alpha(x) : \alpha \mbox{ action function } \}$$
%is bounded, although not necessarily compact because the set of action functions may not be closed, since a limit of Lipschitz continuous functions is not necessarily Lipschitz continuous. Therefore an optimal action function may not exist. Nevertheless, as this set is bounded, for all $\epsilon>0$ there exist an action function $\alpha^{\epsilon}$ such that
%$$v(x) = \sup_\alpha{v_\alpha(x)}\leq v_{\alpha^\epsilon}(x) + \epsilon$$
%Thus, using first (\ref{valpha}) and then (\ref{limv}), we have
%$$\h V_{\alpha^\epsilon} (X^h(0)) = v_{\alpha^\epsilon}(x_0)\geq v(x_0)-\epsilon = \h V^h(X^h(0)) -\epsilon .$$
%This shows that $\alpha_\epsilon$ is optimal up to $2\epsilon$ for $h$ small. 

\subsubsection{Auxiliary results}

In order to prove the main theorems we need two auxiliary results. 
%In the first only policies appear, so it considers the second auxiliary system and the controlled system of small players.% interpolated as in definition \ref{def8}.

\begin{theorem}
Under assumptions  (H1)-(H6), there exist functions $I_0$ and $J$ satisfying $\h I_0(h,\alpha) = \h J(h,T) =0$
such that
%\quad \lim_{\substack{h\rightarrow0 \\ \delta\rightarrow0}}B'(h,\delta)=0.
for any $\epsilon>0$, $h>0$ and any policy $\pi$
\be
P\left\{\sup_{0<t<T}\left\|\hat{X}^h_\pi(t)-X(t,x_0, A_\pi^h)\right\|
>\left[\left\|X^h(0)-x_0\right\|+I_0(h)T +\epsilon\right]e^{L_1 T}\right\}
\leq\frac{J(h,T)}{\epsilon^2}.
\label{T5}
\ee

\label{teo5}
\end{theorem}

 If (\ref{xh}) holds this theorem shows the convergence in probability of the controlled system with $N$ agents,  with explicit bounds. 
%where $J$ is defined by (\ref{J}) and $J(h,T)$ and $I_0(h)$ go to 0 as $h$ goes to 0, for $T$ fixed. As observed above, the result is equivalent to the week convergence.  

The second statement deals with the convergence of the value for the controlled system of small players to the value of the second auxiliary system. Let $\pi$ be a policy and $A_\pi^h$ be the sequence of actions corresponding to a trajectory of $X_\pi^h$, as in notation \ref{def9}. Equation (\ref{val}) defines the value for the deterministic limit, whereas $\alpha$ is an action function. When applying the random action function $A_\pi^h$, this defines a random variable $v_{A_\pi^h}(x_0)$. A consequence of Theorem \ref{teo5} is the convergence of $V_\pi^h(X^h(0))$ to the expectation of this random variable.

\begin{theorem}
Let $A_\pi^h$ be the random action function associated with $X_\pi^h$ as in notation \ref{def9}. 
Under assumptions (H1)-(H6), there exist a function $B$ satisfying
\be
\lim_{\substack{h\rightarrow0 \\ \delta\rightarrow0}}B(h,\delta)=0
\ee
such that
\be
\left|V_\pi^h(X^h(0)) -E[v_{A_\pi^h}(x_0)]\right|\leq B(h,\left\|X^h(0)-x_0\right\|) .
\label{T6}
\ee
%where $B$ is defined by (\ref{B}).
\label{teo6}
\end{theorem}

 This implies that if (\ref{xh}) holds almost surely, respectively in probability, then
\be
\h \left|V_\pi^h(X^h(0)) -E[v_{A_\pi^h}(x_0)]\right| =0
\ee
almost surely, respectively in probability.

%All the proofs of the theorems are given in the next section.

%\subsubsection{Remarks}

The proofs of the main results use the two auxiliary systems. The first auxiliary system provides a strategy for the system with $N$ agents derived from an action function of the mean field limit. It can not do better than the optimal value of the system of small players and it is close to the optimal value of the mean field limit. Therefore the optimal value for the system with $N$ players is lower bounded by the optimal value for the mean field limit.

The second auxiliary system is used in the opposite direction: it shows that for large $N$ the two optimal values are the same.

\subsection{Requirements for convergence}

%In order to prove the above theorems, we apply the method developed in \cite{M}. In that paper the authors obtain the same claims, but they deal with discrete time Markov chains and fixed finite state spaces. Here we consider continuous time Markov chains, but controlled in discrete time, in the infinite dimensional state space $\l$. 
%However their proofs can be adapted to our model. 

Let us denote by $||x||$ the $\ell^2$-norm of $x$ and define the \emph{drift} of the model as
\be
F^h(x,b):=E( X^h(\tau, x, b) -x) .
\label{F}
\ee
Due to Theorems 2-6 in \cite{M}, in order to prove Theorems 1-4 it is sufficient to show that 
%The claims then follows if we prove that
\begin{itemize}
	\item \textbf{(A1)} There exist some non random functions $I_1(h)$ and $I_2(h)$ such that 
	$$\lim_{h\rightarrow0} I_1(h) = \lim_{h\rightarrow0} I_2(h) =0$$
	and that for all $x$ and all policies $\pi$ the number coalitions $\Delta_\pi^h(k)$ that perform a transition between time  step $k\tau$ and $(k+1)\tau$ satisfies
	\begin{eqnarray}
	E(\Delta_\pi^h(k) |X_\pi^h(k\tau)=x) &\leq& \frac{1}{h} I_1(h) \label{delta1}\\
		E(\Delta_\pi^h(k)^2 |X_\pi^h(k\tau)=x) &\leq& \frac{1}{h^2} \tau I_2(h) \label{delta2};
		\end{eqnarray}
%where $I(h)$ is the \emph{intensity function} of the model;

\item \textbf{(A2)} %\emph{Convergence of the drift}
%\noindent
%A function $I(h)$ is said to be the  \emph{intensity function} of the model $X^h$ is there exist a function $I_0(h)$ such that
%\be
%\lim_{h\rightarrow0} I(h) = \lim_{h\rightarrow0} I_0(h) =0
%\label{I0}
%\ee
There exist  a function $I_0(h)$ such that 
 $\lim_{h\rightarrow0} I_0(h) =0$ and
\be
\left\|\frac{F^h(x,b)}{\tau(h)}- f(x,b)\right\|\leq I_0(h)
\label{I}
\ee
 for every $x\in S$ and $b\in E$ where  the function $f$ is the one in  (\ref{effe}), and moreover
$f$ is defined on $S\times E$ and there exists a constant $L_2$ such that 
\be
|f(x,b)|\leq L_2 ;
\label{bound}
\ee

\item \textbf{(A3)} %\emph{Lipschitz Continuity}
There exist constants $L_1$, $K$ and $K_B$ such that for all $x$, $y\in S$ and $a$, $b \in E$ 
\begin{eqnarray}
||F^h(x,b)-F^h(y,b)|| &\leq& L_1||x-y|| I(h) \label{lipF}\\
||f(x,b)-f(y,a)|| &\leq& K(||x-y|| + d(a,b))\label{lipf}\\
||B(x,b)-B(y,b)|| &\leq& K_B||x-y|| \label{lipB}\\
||V_0(x)-V_0(y)|| &\leq& K_B||x-y|| 
\label{lipV0}
\end{eqnarray}
\noindent
and the reward is bounded:
\be
\sup_{\substack{x\in S\\b\in E}} \max\left\{|B(x,b)|, |V_0(x,b)|\right\} =: ||B||_{\infty}.
\label{boundrew}
\ee

\end{itemize}

We will show in the next section that if our assumptions (H1)-(H3) are satisfied then (A1)-(A3) hold.
%Our $\tau(h)$ is called \emph{intensity function} of the model in the setting of \cite{M}; recall that we assume $\tau=\tau(h)$ and it tends to 0.

%\subsubsection{Constants}
Let us now fix the functions appearing in (A1), (A2) and (A3):  
\begin{eqnarray}
K &:=& 6CR + 3F + 3R[C(1)R + F(1)],\\
 L_2 &:=& 3CR^2 + 3FR,\\
I_0(h) &:=& \sqrt{N(h)} \frac{\tau}{2}(R_1 + h R_2),
\label{io}
\\
I_1(h) &:=& \tau(CR^2 +F),
\label{i1}
\\
I_2(h) &:=& (CR^2 +F) [ \tau(CR^2 +F) + h],
\label{i2}
\end{eqnarray}
where
\begin{align*}
R_1 &:= 3(CR^2 + FR)( 6CR +3F +3 C(1)R^2 +F(1)R),\\
R_2&:= 54(CR^2 + FR)(C +F(1) +C(1)R +F(1)R + C(2)R^2).
\end{align*}
Clearly $\h I_0 (h) = \h I_1 (h) =\h I_2 (h) =0$ if (H4) holds.
 $L_1$ actually depends on $h$
\be
L_1 = L_1 (h):=K e^{M_2 \sqrt{N(h)} \tau},
\label{l1}
\ee
where
$M_2 := 3(C(1)R^2 + 2CR + F(1) R +F)$.  
Hence $L_1$ is  tends to the constant $K$ by (H4), as $h$ tends to 0.
%: this makes the proofs still work.

Let us define the functions $I_0',J,B,B'$ appearing in the statements of Theorems 2-4 by the following equations
\begin{align}
I_0'(h,\alpha)&:= I_0(h) +\tau K e^{(K-L_1)T}\cdot\left[\frac{K_\alpha}{2}+2 \left(1+\min \left\{\frac{1}{\tau},p\right\}\right)||\alpha||_\infty\right],\\
J(h,T)&:=8T\left\{L_1^2\left[I_2(h)\tau^2 + I_1(h)^2(T+\tau)\right]+
N(h)^2\left[2I_2(h)+\tau L_2^2\right]\right\},
\label{J}
\\
B(h,\delta) &:= \tau ||B||_{\infty} + K_B \sqrt2 I_1(h) + K_B(\delta +I_0(h)T)\left(e^{L_1 T}+\frac{e^{L_1 T}-1}{L_1}\right)
\label{B}\\
&\quad+\frac{3}{2^\frac13}\left[e^{L_1 T} + \frac{e^{L_1 T}-1+\frac\tau2}{L_1}\right]^\frac23 \cdot K_B^\frac23||B||_\infty^\frac13 J(h,T)^\frac13 (T+1)^\frac23
\end{align}
and $B'(h,\delta)$ has the same expression as $B(h,\delta)$ replacing $I_0(h)$ by $I_0'(h,\alpha)$.
From (H4) and (H6) follow that  $\h J(h,T) = \h  I_0'(h,\alpha) =0$ and 
$\lim_{\substack{h\rightarrow0 \\ \delta\rightarrow0}}B'(h,\delta)= \lim_{\substack{h\rightarrow0 \\ \delta\rightarrow0}}B(h,\delta)=0$.

\subsection{Constructing an optimal policy}

By means of corollary \ref{cor4}, an optimal action function for the mean field limit is asymptotically optimal for the system of small players. This provides a way for constructing an asymptotically optimal policy.

We denote by $u(x,t)$ the optimal cost over horizon $[t,T]$ for the limiting system.
Under our hypothesis, the following  proposition holds.

\begin{proposition}
The value function $u(x,t)$ is the unique, bounded and uniformly continuous, viscosity solution in $S\times [0,T] \subset  \ell^2 \times [0,T]$  of the Hamilton-Jacobi-Bellman equation
\be
-\frac{\partial u(x,t)}{\partial t} - \max_{b\in E} \{ \nabla u(x,t) \cdot f(x,b) + B(x,b) \} = 0
\label{hjb}
\ee
which satisfies the terminal condition $u(x,T) = V_0(x)$.
\label{HJB}
\end{proposition}

 %This is proved for instance in \cite{Bardi} if the dynamics lies in an euclidean space, but the same proof is still valid in our framework where the dynamics lies in the infinite dimensional space $\ell^2$. 
Let us recall that the definition of viscosity solution in a Hilbert space, as $\ell^2$, does not differ from the usual one. 
Further, under our assumptions, the Hamiltonian defined as 
$H(x,p) := \max_{b\in E} \{ p \cdot f(x,b) + B(x,b) \}$ for any $(x,p) \in S\times \ell^2$ is such that
\begin{align*}
| H(x,p) -H(y,p)| &\leq C ||x-y|| (1+||p||)\\
| H(x,p) -H(x,q)| &\leq C ||p-q||, 
\end{align*} 
where $C$ is a constant. 
 Therefore existence and uniqueness of bounded and uniformly continuous viscosity solutions to (\ref{hjb}) are implied by
 Theorem 5.1 in \cite{CL}.
%The optimal cost $u(x,t)$ is usually taken over the set $\mathcal{A}$ of all measurable controls $\alpha$ from $[0,T]$ to the compact metric space $E$. However, the above proposition holds also if we restrict $\mathcal{A}$ to a subset $\mathcal{B}$ with the following properties:
%\begin{itemize}
	%\item if $\alpha \in \mathcal{B}$, $s>0$ then $t\mapsto\alpha(t+s)$ is in $\mathcal{B}$;
	%\item if $\alpha_1, \alpha_2 \in \mathcal{B}$, $s>0$ 
	%and $\alpha(t) := \begin{cases} \alpha_1(t) & t\leq s,\\
  %   \alpha_2(t-s) & t>s,\\
%\end{cases}$

%then $\alpha\in \mathcal{B}$.
%\end{itemize}
%This is remarked in \cite{Bardi} and holds since the proof of proposition (\ref{HJB}) uses the dynamic programming principle. Thus the above proposition is valid whereas we let the value function to be the supremum  over the set of piecewise Lipschitz continuous action function, the one we have always considered from notation (\ref{action}).

%\subsubsection{Algorithm for the policy}

We state  the algorithm presented in \cite{M} for constructing an asymptotically optimal policy for the system of small players $X^h$ via an optimal action function for the mean field limit $X$:

\begin{itemize}
	\item Let $X^h(0)$ be the initial condition of the limiting system. Solve the Hamilton-Jacobi-Bellman equation (\ref{hjb}) on $[0, \tau n(h)]$. Assume this provides an optimal control function $\alpha_\ast $;
	\item Construct a policy $\pi$ for the system of small players: the action to be taken under state $X^h(k\tau)$ at step $k$ is
	$$\pi_k(X^h(k\tau)) := \alpha_\ast ( k\tau) .$$
\end{itemize}

The asymptotic optimality of the related value is ensured by corollary \ref{cor4}.
The policy $\pi$ constructed above is static in the sense that it does not depend on the state $X^h(k\tau)$ but only on the initial state $X^h(0)$. The deterministic estimation of $X^h(k\tau)$ is provided by the differential equation.

%One can construct a more adaptive policy by updating the starting point of the differential equation at each step. This new procedure yields an adaptive policy $\pi'$ whose value should be larger than the value of the static policy $\pi$ because it uses online corrections at each step, before taking a new action. However theorem (\ref{teo2}) does not provide a proof of its asymptotic optimality.

%\subsubsection{Existence of an optimal action function}

The algorithm described above uses an optimal action function for the limiting system which may not exist, as we observed above.  A sufficient condition for its existence is that the set 
%Assume the hypothesis always considered and that the set 
$f(x,E) \times B(x,E)$
 is convex for all $x\in S$; a proof of this fact can be found in \cite{Bardi}. 
%Then there exists $\alpha_\ast \in \mathcal{A}$ such that 
%$$v_{\alpha_\ast}(x) = \max_{\alpha\in } v_{\alpha\in\mathcal{A}}(x) .$$

%This is still valid if we replace $\mathcal{A}$ with the set of action functions, which has the properties previously stated. The proof can be found again in \cite{Bardi}. 

However the existence of an optimal action function is actually not necessary in the algorithm described above. Indeed, if there is no optimal control of the HJB equation (\ref{hjb}), one can replace the optimal $\alpha_\ast$ used in the algorithm by an action function which is $h$-optimal. This still provides an asymptotically optimal policy.

%In \cite{Bardi} there is a procedure in order to find an $\epsilon$-optimal control in feedback form. If the value function, solution of (\ref{hjb}), is smooth and under other hypothesis, the classical verification theorem provides an optimal control. Otherwise, considering the value function as the unique viscosity solution of (\ref{hjb}), one can find an approximated optimal control via the convergence of modified Hamiltonians. 
%This procedure leads also to prove the last lemma.

\subsection{Example}
We show here that in a class of applications the limiting problem can be reduced to an optimization problem in one dimension.
This provides a computationally much more effective scheme in order to obtain an asymptotically optimal policy for the system of small players, in view of the algorithm presented above.

Firstly, we study a particular case in which the mean field limit admits an explicit solution.
So let us consider $E = [0,1]$  and a particular shape of the functions $C_{ij}$ and $F_{ij}$, in which the there is no dependence on $x$: 
\begin{align}
C_{ij} (x,b) &= b\\
F_{ij} (x,b) &= \frac{1-b}{i-1}.
\end{align}
In particular if $b$ is 1 then only merging is possible. While if it is  0 only splitting is possible.

Let $m(x):= \sum_i x_i$ be the norm of $x\in\ell^1$.
With these rates the limiting evolution $f$ in \eqref{effe} can be studied considering only the dynamics of the norm.
Evaluating the limiting generator \eqref{inff} on $G=m$ we obtain
\be
\sum_{i} f_i(x,b) = -b m^2 + (1-b) m .
\ee
Hence the evolution of $m$ is described by
\be
\begin{cases}
\dot{m}=  -b m^2 + (1-b) m\\
m(0) =m_0,
\end{cases}
\label{mbb}
\ee
which is in fact an ODE on $\mathbb{R}_+$.

If $b$ is constant then the explicit solution to \eqref{mbb} is 
\be
m(t,m_0,b) = \frac{m_0(1-b)e^{(1-b)t}}{1-b-m_0 b + m_0 e^{(1-b)t}}.
\ee
if $b\neq 1$. In particular, if $b=0$, $m(t,m_0, 0) = m_0 e^t$. While if $b=1$ 
\be
m(t,m_0, 1) = \frac{m_0}{1+t m_0}.
\ee
Observe that $m(t,m_0,b)$ is a continuous and strictly decreasing function of $b\in [0,1]$, for any value of $0<t\leq T$ and $m_0>0$.

We assume that there is no instantaneous reward and that the final reward $V_0 = V_0 (m)$ is a strictly concave function 
of $m$ which has a unique maximum in $m^\ast$.
The value function $V$ in this case can be computed explicitely, that is 
\be
V(t,m) = \begin{cases}
V_0 (m e^{T-t} ) & \mbox{ if } m< m^\ast e^{-(T-t)}\\
V_0(m^\ast) & \mbox{ if }  m^\ast e^{-(T-t)}\leq m \leq \frac{m^\ast}{1-(T-t)m^\ast}\\
V_0 \left(\frac{m}{1+(T-t)m}\right) &  \mbox{ if } m>\frac{m^\ast}{1-(T-t)m^\ast}.
\end{cases}
\label{mb}
\ee

If $m^\ast e^{-(T-t)}\leq m_0 \leq \frac{m^\ast}{1-(T-t)m^\ast}$ denote by $b^\ast(T,m_0)$ the unique value of $b\in [0,1]$ such that 
$m(T,m_0,b) = m^\ast$. 
Therefore, if $m(0) = m_0$, an optimal action function for the mean field limiting problem is to choose the constant value
\be
\alpha^\ast (t,m_0) =  \alpha^\ast(m_0)  :=
\begin{cases}
1 & \mbox{ if }  m_0 <  m^\ast e^{-(T-t)}\\
b^\ast(T,m_0) &  \mbox{ if } m^\ast e^{-(T-t)} \leq m_0 \leq \frac{m^\ast}{1-(T-t)m^\ast}\\
0 & \mbox{ if }  m_0 > \frac{m^\ast}{1-(T-t)m^\ast}.
\end{cases}
\ee

This example can be generalized to the case in which the rates $C_{ij}$ and $F_{ij}$ depend also on the norm of $x$. Namely, let
\begin{align}
C_{ij} (x,b) &= b f_C (m(x))\\
F_{ij} (x,b) &= \frac{1-b}{i-1} f_B (m(x)),
\end{align}
where $f_C$ and $f_B$ are some Lipschitz continuous and non-negative functions. 
The dynamics (\ref{mbb}) for $m$ becomes
\be
\begin{cases}
\dot{m}=  -b f_C (m)m^2 + (1-b) f_B (m) m\\
m(0) =m_0.
\end{cases}
%\label{mbb}
\ee

Also in this case, the  mean field limiting problem  reduces to an optimization problem in one dimension. Thus there are numerical schemes which provide an $\frac{1}{N}$-optimal action function in feedback form (see e.g. \cite{Bardi}). These are much more efficient than trying to solve the Bellman equation (\ref{bell}). In fact the prelimit problem is allowed to be tackled when $N$ is lower than a few tens:
see \cite{ten}.

%Therefore if $m(0) = m_0$ then an optimal action function is
%\be
%\alpha^\ast (t) = 
%\begin{cases}
%1 & m_0< m^\ast e^{-(T-t)}\\
%%\frac{1}{m_0+1} & m^\ast e^{-(T-t)}\leq m_0 \leq m^\ast \mbox{ and } 0\leq t \leq T + \log\left(\frac{m_0}{m^\ast}\right)\\
%0 & m_0>\frac{m^\ast}{1-(T-t)m^\ast}\\
%1 &  m^\ast e^{-(T-t)}\leq m_0 \leq m^\ast \mbox{ and } 0\leq t \leq T + \log\left(\frac{m_0}{m^\ast}\right)
%\end{cases}
%\ee

\section{Proofs of convergence}

%This section is our main contribution. 
In this section we show that (A1)-(A3) hold in our fragmentation-coagulation model.
%ssumptions stated in section 1.4 hold in our fragmentation-coagulation model. Of course we have to assume some regularity on the functions involved in the model. The conditions needed will be summarized in the end of the section. First of all we need to define an appropriate compact subset $S$ of $\ell^1=\ell^1(\mathbb{N})$ where both the dynamic of the system of small players and of the mean field limit lie. 
So  (H1)-(H6) hold and recall that the compact state space $S\subset \ell^1$ is stable for all the dynamics  considered.

\subsection{Lipschitz continuity of the limit } 

We  verify that (\ref{lipf}) and (\ref{bound}) are satisfied for the function $f$ in (\ref{effe}), if (H2) holds. We actually do not need $f$ to be in $\mathcal{C}^2 (S)$, but only in   $\mathcal{C}^1 (S)$. We use the fact that  
\be
||x|| =||x||_{\ell^2}\leq ||x||_{\ell^1}\leq ||x||_{\ell^1(L)}\leq R
\label{catena}
\ee
for any $x\in S$. 
So equation (\ref{uno}) gives 
$$||f(x)||\leq 3 C R^2 +3FR $$
for all $x\in S$, which is the assumption (\ref{bound}) with bound 
\be
L_2:= 3 C R^2 +3FR.
\label{L2}
\ee

By (\ref{catena}) and its definition (\ref{der}) we get also
$$||Df(x)||_{\ell^2} \leq ||Df(x)||_{\ell^1} \quad \forall x\in S$$
which, together with equation (\ref{due}), yields
\be
||Df(x)||_{\ell^2}\leq 6 C R +3 F  +3R [C(1) R +F(1)].
\label{boundder}
\ee

Applying then the mean  value theorem in the convex space $S$, for every $x$ and $y$ in $S$  there exist a point $z$ in the segment $[x,y]$ such that 
$$f(y) - f(x) =  Df(z).(y-x) .$$
Thus taking the $\ell^2$-norm and using the fact that $Df$ is a bounded linear map we get
$$||f(y) - f(x)|| \leq  ||Df(z)||_{\ell^2} ||y-x||$$
for all $x$ and $y$ in $S$. According to (\ref{boundder}) we find that the limiting function is Lipschitz continuous in $x$, for the 
$\ell^2$-norm with 
\be
K= 6 C R +3 F  +3R [C(1) R +F(1)]
\label{K}
\ee
as a Lipschitz constant.
%$$||f(y) - f(x)|| \leq  K||y-x||.$$

%If we also assume that the functions $C_{ij}(x,b)$ and $F_{ij}(x,b)$ are Lipschitz continuous in $b$ then  equation (\ref{lipf}) in the hypothesis (A3) is satisfied. 

%\subsection{Assumptions on the reward}

%The hypothesis (A3) says that, in order to get convergence, the rewards 
%$$B:S\times E \longrightarrow \R^+$$
%and the final reward 
%$$V_0:S\longrightarrow \R^+$$
%have to be bounded and Lipschitz continuous in $x$ in the $\ell^2$-norm, for any $b$. This is true, for instance, if they are in 
%$\mathcal{C}^1 (S)$, as observed above, since $S$ is compact and so they have bounded derivatives.
Clearly (H3) implies (\ref{lipB}), (\ref{lipV0}) and (\ref{boundrew}).

\subsection{Convergence of the drift}

Here we  prove (\ref{I}).
Recall that the drift of the model is
$$F^h(x,b) := E[X^h(\tau,x,b) -x].$$
The semigroup $U^h_t$ of the Markov process $X^h$ is  defined by
$$U^h_t G (x,b) := E[G(X^h(t,x,b))]$$
for any $G\in \mathcal{C}(S, \R)$ and $t\geq0$. $U_0$ is the identity and $U_t$ satisfies the Kolmogorov differential equation
\be
\frac{\partial}{\partial t } U^h_t G (x,b) = \Lambda_{h,b} U^h_t G (x,b) =
U^h_t \Lambda_{h,b} G (x,b) ,
\label{DUL}
\ee
where $\Lambda_{h,b}$ is the infinitesimal generator defined in (\ref{gen}).

We apply $U^h_t$ to the projection on the $k$-th coordinate $G_k$. This function is  in $\mathcal{C}(S, \R)$ and  
\be
F^h_k(x,t) := G_k\left(E[X^h(\tau,x,b) -x]\right) = E G_k(X^h(\tau,x,b)) -x_k = U^h_\tau G_k (x,b) - G_k(x).
\label{Fhk}
\ee

Let us denote 
$u_k(t,x,b):= U^h_t G_k (x,b).$
The Taylor formula applied to $u_k$ in the variable $t$ gives
\be
u_k(\tau,x,b) = u_k(0,x,b) + \tau \frac{\partial u_k}{\partial t}(0,x,b) + \frac{\tau^2}{2} \frac{\partial^2 u_k}{\partial t^2}(s,x,b)
\label{taylor}
\ee
for a fixed $s\in ]0,\tau[$ and for every $x$ and $b$.
We have
\be
u_k(0,x,b)= G_k(x) = x_k
\label{U0}
\ee
 and according to  (\ref{DUL}) 
\be
\frac{\partial u_k}{\partial t}(0,x,b) =\Lambda_{h,b} u_k (0,x,b) = \Lambda_{h,b} G_k(x)
\label{U'0}
\ee
for every $x$ and $b$.

The generator (\ref{gen}) calculated  on the projection, using $G_k(e_i) = \delta_{ik}$, gives 
\begin{align*}
\Lambda_{b,h} G_k(x) &= \frac1h \sum_{i,j} C_{i j}(x,b) x_i x_j \left[G_k(x-h e_i - h e_j+h e_{i+j}) -G_k(x)\right] 
\\
&+ \frac{1}{h}\sum_i \sum_{i<j} F_{ij}(x,b) x_i \left[G_k(x-h e_i + h e_j+h e_{i-j}) -G_k(x)\right]\\
&=  \sum_{i,j} C_{i j}(x,b) x_i x_j \left[\delta_{i+j,k}- \delta_{ik}- \delta_{jk}\right] \\
&+ \sum_i \sum_{i<j} F_{ij}(x,b) x_i \left[\delta_{jk}+\delta_{i-j,k}- \delta_{ik}\right]\\
&=\sum_{i<k} C_{i,k-i} (x,b) x_i x_{k-i} -\sum_j C_{kj} (x,b) x_k x_j 
-\sum_i C_{ik} (x,b) x_i x_k \\
&+  \sum_{i>k} F_{ik}(x,b)x_i + \sum_{i>k} F_{i,i-k}(x,b)x_i-\sum_{j<k}F_{kj}(x,b)x_k 
\end{align*}
and, observing that $C_{ij}= C_{ji}$ and $F_{ij} = F_{i,i-j}$, the latter expression is exactly $f_k(x,b)$ defined in (\ref{effe}), whence
\be
\Lambda_{b,h} G_k(x) = f_k(x,b)
\label{landaf}
\ee
for every $x$ and $b$.

We need also an estimate of the latter term in (\ref{taylor}). Equation (\ref{DUL}) yields 
$$\frac{\partial^2 u_k}{\partial t^2}(s,x,b) %= \frac{\partial^2 }{\partial t^2}U_s^h G_k(x,b)=
\frac{\partial}{\partial t} U_s^h \Lambda_{b,h} G_k (x)= 
U_s^h \Lambda_{b,h} \Lambda_{b,h} G_k(x) 
=U_s^h \Lambda_{b,h} f_k(x,b).$$
We use then the fact that the semigroup $U_s^h$ is, for any $s$, a contraction in the space $\mathcal{C}(S, \R)$ equipped with the sup-norm $||G||_{\infty}:= \sup_{x\in S} |G(x)|$.  Thus the latter term in (\ref{taylor}) is bounded by 
\be
\left|\frac{\partial^2 u_k}{\partial t^2}(s,x,b)\right|
\leq \left\|U_s^h \Lambda_{b,h} f_k\right\|_{\infty} 
\leq \left\|\Lambda_{b,h} f_k\right\|_{\infty} 
\label{U''}
\ee
for any $x\in S$, $b\in E$ and $s>0$.

The estimate for the norm of $\Lambda_{b,h} f_k$ is found applying again the Taylor formula to $f_k$:
\begin{align*}
\Lambda_{b,h}& f_k(x) = \frac1h \sum_{i,j} C_{i j}(x,b) x_i x_j \left[f_k(x-h e_i - h e_j+h e_{i+j}) -f_k(x)\right] \\
&+ \frac{1}{h}\sum_i \sum_{j<i} F_{ij}(x,b) x_i \left[f_k(x-h e_i + h e_j+h e_{i-j}) -f_k(x)\right]
\\
&= \sum_{i,j} C_{i j}(x,b) x_i x_j \sum_l \frac{\partial f_k}{\delta x_l} (x) \left(\delta_{i+j,l}- \delta_{il}- \delta_{jl}\right) \\
&+h \sum_{i,j} C_{i j}(x,b) x_i x_j    \sum_l \sum_m \frac{\partial }{\delta x_l} \frac{\partial }{\delta x_m} f_k(y) \left(\delta_{i+j,l}- \delta_{il}- \delta_{jl}\right)\left(\delta_{i+j,m}- \delta_{im}- \delta_{jm}\right)
\\
&+ \sum_i \sum_{j<i} F_{ij}(x,b) x_i \sum_l \frac{\partial f_k }{\delta x_l} (x) \left(\delta_{jl}+\delta_{i-j,l}- \delta_{il}\right)\\
&+ h\sum_i \sum_{j<i} F_{ij}(x,b) x_i \sum_l \sum_m\frac{\partial }{\delta x_l} \frac{\partial }{\delta x_m}f_k(z)\left(\delta_{jl}+\delta_{i-j,l}- \delta_{il}\right)\left(\delta_{jm}+\delta_{i-j,m}- \delta_{im}\right)
\end{align*}
for certain fixed points $y$ and $z$ in $S$.
This gives, using the estimates (\ref{CF}), (\ref{due}) and (\ref{tre}) and the definitions of the norms of the derivatives $Df$ and $D^2 f$,
\begin{align*}
\left|\Lambda_{b,h} f_k(x)\right| &\leq C ||x||_{\ell^1}^2 \left(3 ||Df(x)||_{\ell^1} + 9h ||D^2f(y)||_{\ell^1}\right)
\\
 &+  F ||x||_{\ell^1} \left(3 ||Df(x)||_{\ell^1} + 9h ||D^2f(z)||_{\ell^1}\right)
\\
&\leq C ||x||_{\ell^1}^2  [ 3\left(6 C ||x||_{\ell^1} +3 F  +3 [C(1) ||x||_{\ell^1} +F(1)]||x||_{\ell^1}\right)
\\
 &+54h \left(C +F(1) +[C(1) +F(2)]||y||_{\ell^1} +C(2) ||y||_{\ell^1}^2\right)]
\\
 &+F ||x||_{\ell^1} [ 3\left(6 C ||x||_{\ell^1} +3 F  +3 [C(1) ||x||_{\ell^1} +F(1)]||x||_{\ell^1}\right)
\\
 &+54h \left(C +F(1) +[C(1) +F(2)]||z||_{\ell^1} +C(2) ||z||_{\ell^1}^2\right)]
\\
&\leq (CR^2 + FR)[ 3( 6CR +3F +3 C(1)R^2 +F(1)R) 
\\
&+ 54h(C +F(1) +C(1)R +F(1)R + C(2)R^2)]
\end{align*}
for every $x\in S$ and $b\in E$ and $k =1,\ldots, N(h)$.

Therefore 
\be
\left\|\Lambda_{b,h} f_k\right\|_{\infty} \leq R_1 + h R_2
\label{landak}
\ee
for every $x\in S$ and $b\in E$ and $k =1,\ldots, N(h)$, where
$R_1 := 3(CR^2 + FR)( 6CR +3F +3 C(1)R^2 +F(1)R)$ and 
$R_2:= 54(CR^2 + FR)(C +F(1) +C(1)R +F(1)R + C(2)R^2)$.

Equations  (\ref{Fhk}) and (\ref{taylor}), by means of (\ref{U0}), (\ref{U'0}) and (\ref{landaf}),
lead to
$$F_k^h(x,b) - \tau f_k(x,b) = \frac{\tau^2}{2} \frac{\partial^2 u_k}{\partial t^2}(s,x,b)$$
and this applying (\ref{U''}) and (\ref{landak}) yields
$$\left|F_k^h(x,b) - \tau f_k(x,b)\right| \leq \frac{\tau^2}{2} (R_1 + h R_2)$$
for any $k=1,\ldots, N(h)$. 
Considering  the $\ell^2$-norm we have
%$$\left\|F^h(x,b) - \tau f(x,b)\right\| \leq \sqrt{N(h)} \frac{\tau^2}{2}(R_1 + h R_2)$$
 hence
$$\left\|\frac{F^h(x,b)}{\tau} -  f(x,b)\right\| \leq   \sqrt{N(h)} \frac{\tau}{2}(R_1 + h R_2)$$
which is (\ref{I}), whereas $I_0$ is given by (\ref{io}).

%Therefore assumption is satisfied if we let 
%$$I(h)= \tau$$
%to be the intensity function and
%$$I_0(h) := \sqrt{N(h)} \frac{\tau}{2}(R_1 + h R_2).$$
%Thus we must assume that $\tau = \tau(h)$ depends on $h$ and  
%\be
%\h \tau(h) =0 \quad \mbox{and} \quad \h \tau(h) \sqrt{N(h)} =0
%\label{tauh}
%\ee
%so that $\h I_0(h) =\h I(h) =0$.

%In practice in this way we study the limit as both $h$ and $\tau$ tend to 0. Moreover all the optimization problems for the systems of small players have now a finite horizon time $\tau(h) n(h)$ less and closed to $T$, the fixed horizon time for the mean field limit, since we defined $n(h) = \left\lfloor T/I(h)\right\rfloor$. So in the following we will always write $\tau=\tau(h)=I(h)$.

\subsection{Lipschitz continuity of the drift}

Here we verify that the drift $F^h$ of the model is Lipschitz continuous and we also find a constant for which it is bounded in $\ell^2$.
We use two tools. The first is the fact that the process

%\begin{lemma}
%Let $X(t)_{t\geq0}$ be a Markov process on a compact metric space $S$ with strongly continuous semigroup generated by the operator $\Lambda: \mathcal{D} \subseteq \mathcal{C}(S)\longrightarrow \mathcal{C}(S)$. Then for every $G\in \mathcal{D}$ the process
\be
M_G(t):= G(X(t)) - G(X(0)) - \int_0^t \Lambda (X(s)) ds
\label{mart}
\ee
 are martingales with respect to the filtration generated by the Markov process $X(t)_{t\geq0}$, for every $G$ in the domain of its generator.
%\end{lemma} 
While
the second tool is the notion of \emph{coupling} for Markov chains. 

%Let us recall the simple definition of couplings.

%\begin{definition}
%Let $\mu_1$ , respectively $\mu_2$, be a probability on a measurable space $(E_1, \mathcal{E}_1)$, respectively $(E_2, \mathcal{E}_2)$. A probability measure $\tilde{\mu}$ on the product space $(E_1\times E_2, \mathcal{E}_1 \otimes\mathcal{E}_2)$ is called a \emph{coupling} of $\mu_1$ and $\mu_2$ if the following \emph{marginality} condition holds:
%$$\tilde{\mu} (A_1 \times E_2) = \mu_1(A_1) \qquad \forall A_1 \in \mathcal{E}_1,$$
%$$\tilde{\mu} (A_1 \times E_2) = \mu_1(A_1) \qquad \forall A_1 \in \mathcal{E}_1.$$

%\end{definition}

We want to study the behavior of the two Markov chains $X^h(t,x,b)$ and $X^h(t,y,b)$ for $x\neq y$ and link them in some sense in the product space. So we could  define a coupling of these stochastic processes in terms of their distributions in the space of paths $D([0,T],S)$, the space of cadlag functions, for fixed initial points. However, for given marginal Markov processes, the resulting coupled process may not be Markovian. So we introduce the following fundamental definition, which can be found for instance in \cite{chen}.

\begin{definition}
Given two Markov processes with semigroups $U_j(t)$ and generators $\Lambda_j$, or transition probabilities $P_j(t,x_1, \cdot)$, on
$(E_j, \mathcal{E}_j)$, $j=1,2$, 
a \emph{Markovian coupling} is a Markov process with semigroup $\tilde{U}(t)$ and generator $\tilde{\Lambda}$, or transition probability
$\tilde{P}(t;x_1,x_2,\cdot)$, on the product space $(E_1\times E_2, \mathcal{E}_1 \otimes\mathcal{E}_2)$ having the marginality:
\begin{itemize}
	\item
$$\tilde{P}(t;x_1,x_2,A_1 \times E_2) = P_1(t,x_1, A_1)\quad\forall t\geq 0, x_1\in E_1, A_1 \in\mathcal{E}_1,$$
$$\tilde{P}(t;x_1,x_2,E_1 \times A_2) = P_2(t,x_2, A_2)\quad\forall t\geq 0, x_2\in E_2, A_2 \in\mathcal{E}_2.$$
Or equivalently
$$\tilde{U}(t) G(x_1,x_2) = U_1(t)G(x_1), \quad\forall t\geq 0, x_1\in E_1, G\in B(\mathcal{E}_1),$$
$$\tilde{U}(t) G(x_1,x_2) = U_2(t)G(x_2), \quad\forall t\geq 0, x_2\in E_2, G\in B(\mathcal{E}_2),$$
where $B(\mathcal{E})$ is the set of all bounded $\mathcal{E}$-measurable functions. Here on the left hand side $G$ is regarded as a bivariate function, although it depends on only one variable.

\item 
$$\tilde{\Lambda}G(x_1,x_2) = G(x_1) \quad\forall x_1\in E_1, G\in B(\mathcal{E}_1),$$
$$\tilde{\Lambda}G(x_1,x_2) = G(x_2) \quad\forall x_2\in E_2, G\in B(\mathcal{E}_2),$$
where $G$ on the left hand side is regarded as above.

\end{itemize}
\label{coupling}
\end{definition}

The main result concerning coupling of Markov chains, whose proof can be found  in \cite{chen}, is that the two parts given in the definition are equivalent under certain conditions, for instance if the two Markov chains take value in a finite state space.
A Markovian coupling always exists: the simplest is the \emph{independent coupling}. Consider a finite set $E = E_1=E_2$,
in the notations of the above definition. The generator of the independent coupling is defines by 
$$\tilde{\Lambda}G(x_1,x_2) := \Lambda_1 G(\cdot,x_2) (x_1) + \Lambda_1 G(x_1 ,\cdot) (x_2)$$
for any $G\in \mathcal{C} (E^2)$.

% Whether this equivalence is valid also for diffusions is still an open problem.

%\begin{proposition}
%Consider two Markov chains, then any coupling of them is a Markov chain and the following hold:
%\begin{enumerate}
	%\item if a coupling operator is non-explosive, then so are its marginals;
	%\item if the marginals are both non-explosive, then so is every coupling operator;
	%\item in the non-explosive case, the two parts of the definition (\ref{coupling}) are equivalent.
%\end{enumerate}
%\label{prop}
%\end{proposition}

Hence, for fixed $h$ and $b$ (which will be omitted in the following writing) we consider \emph{any} coupling operator $\widetilde{\Lambda}_h$ of the Markov chain $X^h(t)$ and itself, which gives the semigroup $\tilde{U}^h_t$.  The coupled process will be called $(X^h(t),Y^h(t))$, although we consider the same process, to avoid misunderstandings. 
%$X^h(t)$ takes values in the finite set $S(h)$ and is non-explosive, thus the hypothesis of both lemma (\ref{mart}) and proposition (\ref{prop}) are satisfied. 
Then by  (\ref{mart}) the process 
\be
M_g(t):= g(X^h(t),Y^h(t)) - g(X^h(0),Y^h(0)) - \int_0^t \tilde{\Lambda}g (X^h(s),Y^h(s)) ds 
\label{Mg}
\ee
 is a martingale for any $g\in  \mathcal{C}(S(h)\times S(h))$.

We let 
$$E^x[G(X^h(t))] := E[G(X^h(t))|X(0)=x] = E[G(X^h(t,x))]= U^h_t G(x)$$
 and 
$$\tilde{E}^{x,y}[g(X^h(t),Y^h(t))] := \tilde{U}^h_tg(x,y) = \tilde{E}\left[g(X^h(t),Y^h(t))|(X^h(0),Y^h(0))=(x,y)\right]$$ 
for any $x$ and $y$ in $S(h)$.
Therefore the martingale $M_g$ defined in (\ref{Mg}) leads to
\be
\tilde{E}^{x,y}[g(X^h(t),Y^h(t))] - g(x,y) = \tilde{E}^{x,y} \left[\int_0^t \tilde{\Lambda} g(X^h(s),Y^h(s))ds\right]
\label{Eg}
\ee
for any $t\geq0$ and $g\in  \mathcal{C}(S(h)\times S(h))$.

Considering $g_k(x,y):= x_k-y_k$ and $G_k(x)=x$,
by means of definition \ref{coupling} and identity (\ref{landaf}), we have
$$\tilde{\Lambda} g_k(x,y) = \Lambda G_k(x) - \Lambda G_k(y) = f_k(x) - f_k(y)$$
and
$$\tilde{E}^{x,y}[g_k(X^h(\tau),Y^h(\tau))] - g_k(x,y) = E^x[X^h_k(\tau))] - x - E^y[X^h_k(\tau))]+y = F^h_k(x)-F^h_k(y)$$
	by definition (\ref{F}) of the drift. 
Hence equation (\ref{Eg}) yields
$$F^h(x) -F^h(y)= \tilde{E}^{x,y}\left[\int_0^\tau \left(f(X^h(s))-f(Y^h(s))\right) ds\right]$$
which, taking the $\ell^2$-norm and applying Fubini's theorem and the Lipschitz continuity of $f$ (\ref{lipf}), becomes
\be
||F^h(x) -F^h(y)||\leq K \int_0^\tau \tilde{E}^{x,y}\left\|X^h(s)-Y^h(s)\right\| ds
\label{Fnorm}
\ee
for any coupling operator $\tilde{\Lambda}_h$, where $K$ is defined in (\ref{K}).

Now we need the following lemma, which is stated for instance in \cite{chen2}:

\begin{lemma}
Let $U_t$ be a strongly continuous semigroup on $E$ with generator $\Lambda$ whose domain is $\mathcal{D}$, $\alpha \in \R$ be a constant and $f\in \mathcal{D}$.
Then $U_t f \leq e^{\alpha t}f$ if and only if $\Lambda f \leq \alpha f$.
\label{lemrho}
\end{lemma}
	
\begin{proof} 

($\Leftarrow$) Let $\Lambda f \leq \alpha f$, then by Kolmogorov's equation
$$\frac{d}{dt}U_t f = \Lambda U_t f \leq \alpha U_t f .$$
Thus by Gronwall's lemma
$$U_t f \leq e^{\alpha t}U_0 f = e^{\alpha t}f .$$

($\Rightarrow$) Let $U_t f \leq e^{\alpha t}f$, then by definition of the generator
$$\Lambda f = \lim_{t\rightarrow0} \frac{U_t f - f}{t} \leq 
\lim_{t\rightarrow0} \frac{e^{\alpha t} - 1}{t} f = \alpha f .$$

\end{proof}	
	
Hence we want to apply this lemma to the coupling operator $\tilde{\Lambda}_h$ and the function $\rho(x,y) = ||x-y||$, in order to obtain an upper bound of the right hand side in (\ref{Fnorm}).
So we have to choose a particular coupling for which there exists $\alpha \in \R$ such that the condition 	
\be
\tilde{\Lambda}_h\rho \leq \alpha \rho
\label{rho}
\ee
is satisfied.

We use  the so called \emph{coupling of marching soldiers}, introduced by Chen in 1986 and whose description can be found for instance in \cite{chen}. This gives a Markov chain in $S(h)\times S(h)$ such that, if it is in $(x,y)$, it jumps to
\begin{itemize}
	\item $(x+z,y+z)$ at rate $\min\{q(x,x+z), q(y,y+z)\}$,
	\item $(x+z,y)$ at rate $\left[q(x,x+z)- q(y,y+z)\right]^+$,
	\item $(x,y+z)$ at rate $\left[q(x,x+z)- q(y,y+z)\right]^-$,
\end{itemize}
for any $z\in S(h)$, where $q(x,y)$ are the rates, i.e. the elements of the Q-matrix, of the Markov chain $X^h$.
It is a Markov coupling, i.e. it satisfies definition \ref{coupling} part 2, since
$\min\{a,b\} + (a-b)^+ =a$ for any  $a,b \geq0$.

Set $h_{ij}=-h e_i - h e_j+h e_{i+j}$ and $h^{ij}=-h e_i + h e_j+h e_{i-j}$.  Thus the generator of the marching coupling $\tilde{\Lambda}$ of the Markov chain $X^h$, whose generator is defined in (\ref{gen}), is given by
\begin{align}
\tilde{\Lambda}_h &g(x,y) = \frac1h \sum_{i,j} \min\{C_{i j}(x) x_i x_j,C_{i j}(y) y_i y_j\} 
\left[g(x +h_{ij},y+h_{ij}) -g(x,y)\right] 
\label{marchg}
\\
&+
\frac1h \sum_{i,j} \left[C_{i j}(x) x_i x_j- C_{i j}(y) y_i y_j\right] ^+
\left[g(x+h_{ij},y) -g(x,y)\right]\nonumber
 \\
&+
 \frac1h \sum_{i,j} \left[C_{i j}(x) x_i x_j - C_{i j}(y) y_i y_j\right]^- 
\left[g(x,y+h_{ij}) -g(x,y)\right]\nonumber
 \\
&+
\frac{1}{h}\sum_i \sum_{j<i} \min\{F_{ij}(x) x_i, F_{ij}(y) y_i \} \left[g(x+h^{ij}, y+h^{ij}) -g(x,y)\right]\nonumber
\\
&+
\frac{1}{h}\sum_i \sum_{j<i} \left[F_{ij}(x) x_i- F_{ij}(y) y_i \right]^+ \left[g(x+h^{ij}, y) -g(x,y)\right]
\nonumber
\\
&+
\frac{1}{h}\sum_i \sum_{j<i} \left[F_{ij}(x) x_i- F_{ij}(y) y_i \right]^- \left[g(x, y+h^{ij}) -g(x,y)\right]
\nonumber
\end{align}
for any $g\in\mathcal{C}(S(h)\times S(h))$.
 Hence, calculating this generator (\ref{marchg}) in the distance function $\rho(x,y) = ||x-y||$, 
 we obtain
\be
\tilde{\Lambda}_h \rho(x,y) \leq 3 \sum_{i,j} \left|C_{i j}(x) x_i x_j- C_{i j}(y) y_i y_j\right|
+ 3\sum_i \sum_{j<i} \left|F_{ij}(x) x_i- F_{ij}(y) y_i \right|,
\label{marchrho}
\ee
using the identity $a^+ + a^- = |a|$ for any real number $a$ and the upper bound
$||x - y-z|| - ||x-y|| \leq ||z|| \leq 3h$ whereas $z$ can be either $-h e_i - h e_j+h e_{i+j}$ or $-h e_i + h e_j+h e_{i-j}$.

In order to get an estimate of the above equation, we shall consider the two functions 
$u,v : S\rightarrow \R^{N\times N} \cong \R^{N^2}$, where $N=N(h)$, defined by
\begin{eqnarray}
u_{ij} (x) &:=& C_{i j}(x) x_i x_j \\
v_{ij} (x) &:=& F_{ij} (x) x_i \mathbb{I}_{]0,+\infty[} (i-j).
\label{uv}
\end{eqnarray}
So the right hand side in (\ref{marchrho}) is equal to 
$$3 ||u(x) -u(y)||_{\ell^1} +3 ||v(x) -v(y)||_{\ell^1}.$$
 
The derivatives of $u$ and $v$ are given by
$$\frac{\partial u_{ij}}{\partial x_k} (x) = \frac{\partial C_{i j}(x)}{\partial x_k}x_i x_j
+ C_{i j}(x) \delta_{ik} x_j +C_{i j}(x) x_i \delta_{jk}  $$
and
$$\frac{\partial v_{ij}}{\partial x_k} (x) = \left(\frac{\partial F_{i j}(x)}{\partial x_k} x_i 
+ F_{i j}(x) \delta_{ik}\right)\mathbb{I}_{]0,+\infty[} (i-j).$$
Thus we apply the mean value theorem to get
\begin{align*}
||u(x) -u(y)&||_{\ell^1} + ||v(x) -v(y)||_{\ell^1} \\
&= \left\|\frac{\partial u }{\partial x} (z). \xi\right\|_{\ell^1}
+ \left\|\frac{\partial v }{\partial x} (w). \xi\right\|_{\ell^1} \\
&\leq \sum_{ijk} \left|\frac{\partial u_{ij}}{\partial x_k} (z) \right| |\xi_k| +
\sum_{ik} \sum_{j<i} \left|\frac{\partial v_{ij}}{\partial x_k} (w) \right| |\xi_k| \\
&\leq C(1) ||z||_{\ell^1}^2 ||\xi||_{\ell^1} + 2 C ||z||_{\ell^1} ||\xi||_{\ell^1} + F(1) ||w||_{\ell^1} ||\xi||_{\ell^1} + F||\xi||_{\ell^1}
\end{align*}
for any $x,y \in S(h)$ and for certain $z$ and $w$ in $S$, where $\xi = x-y$. The latter inequality follows from  (\ref{CF}) and (\ref{CF1}). 

Therefore   (\ref{marchrho}) becomes
$$\tilde{\Lambda}_h\rho(x,y) \leq 3(C(1)R^2 + 2CR + F(1) R +F) ||x-y||_{\ell^1}$$
for any $x$ and $y$ in $S(h)$. 
If we use the estimate $||x-y||_{\ell^1}\leq \sqrt{N} ||x-y||$ for any $x$ and $y$ in $S(h)$ then
\be
\tilde{\Lambda}_h\rho(x,y) \leq M_2 \sqrt{N(h)} ||x-y||,
\label{landarho}
\ee
where $M_2 := 3(C(1)R^2 + 2CR + F(1) R +F)$ is constant, which says that (\ref{rho}) holds with 
$\alpha:= M_2 \sqrt{N(h)} >0.$

Thus we can apply lemma \ref{lemrho} to the marching coupling, so that 
$$ \tilde{E}^{x,y}\left\|X^h(s)-Y^h(s)\right\| \leq e^{\alpha s} ||x-y||$$
for any $s\in [0,\tau]$ and $x\neq y \in S(h)$. Hence (\ref{Fnorm}) becomes
$$||F^h(x) -F^h(y)||\leq K \int_0^\tau e^{\alpha s} ||x-y|| ds \leq K\tau e^{\alpha \tau} ||x-y||,$$
which is (\ref{lipF}) where $L_1$ is the function defined in (\ref{l1}):
$$||F^h(x) -F^h(y)|| \leq K \tau e^{M_2 \tau\sqrt{N(h)} } ||x-y||.$$

%Thus the hypothesis (\ref{lipF}) is actually not true. Nevertheless if we assume that $\tau$ satisfies
%$$
%\h \tau(h) \sqrt{N(h)} = 0,
%$$
%which is the same hypothesis (\ref{tauh}) that we made in the previous section, then we have
%\be
%\h K e^{M_2 \sqrt{N(h)} \tau} = K,
%\ee
%which is a constant. And all the proofs that we will see in the next chapter are still working if we allow
%$$L_1 = L_1 (h):=K e^{M_2 \sqrt{N(h)} \tau},$$
%which tends to the constant $K$, as $h$ tends to 0.

\subsubsection{Boundness of the drift}

Applying  (\ref{mart}) simply to the process $X^h$, without considering couplings, we obtain
\be
F^h(x,b) = E\left[\int_0^\tau f(X^h(s,x,b))ds\right]
\label{FE}
\ee

Hence from (\ref{FE}), by means of (\ref{bound}), we get also an upper bound for the drift:
\be
||F^h(x,b)||\leq L_2 \tau
\label{boundFh}
\ee
for every $x\in S(h)$ and $b\in E$, where $L_2$ is defined in (\ref{L2}).

\subsection{Bounds for $\Delta$}

We  find an estimate for $E(\Delta_\pi^h(k) |X_\pi^h=x)$ where $\Delta_\pi^h(k)$  
is the number of coalitions  that perform a transition between time  step $k\tau$ and $(k+1)\tau$. Because of Markovianity the above expectation is independent of $k$, so we can suppose $k=0$. 
Hence we consider $E(\Delta^h|X^h(0)=x_0)$, where $\Delta^h$ is the number of coalitions that change their state between 0 and $\tau$.

If the system is in $x_0$ in $t=0$ there is an exponential clock of parameter $s(x_0,b)$ such that, when it clicks, the system changes its state, say it  goes in $x_1$. Now there is an other exponential clock of parameter $s(x_1,b)$ such that, when it clicks, the system changes its state, say it  goes in $x_2$. We repeat  this procedure until we arrive at time $\tau$. Note that $\Delta^h$ is then less or equal than the number of clicks that we get from 0 to $\tau$.

Thus to estimate $\Delta^h$ we take an upper bound of $s(x,b)$ defined in (\ref{sum})
$$s(x,b)= \sum_{i ,j} n_i n_j h C_{ij}(x,b) + \sum_i n_i F_{ij}(x,b) ,$$
for any $x\in S(h)$.
Using assumption (\ref{CF}) on $C_{ij}$ and $F_{ij}$ we have 
$$s(x,b)\leq C h \sum_{i ,j} n_i n_j   + F \sum_i n_i = C h \left(\sum_i n_i\right)^2 +F \sum_i n_i $$
for any $x$ and $b$.
According to (\ref{NhR}) 
$\sum_i n_i \leq N(h) \leq \frac Rh ,$
which gives 
$$s(x,b) \leq \frac1h (CR^2 +FR) .$$
Hence for any $x\in S(h)$ and $b\in E$, $s(x,b)$ is bounded by a constant $s^h$ that depends only on $h$:
\be
s^h := \frac1h (CR^2 +F). 
\label{sh}
\ee

For a constant $s$ the number of occurrences of clicks is known to be a Poisson process of intensity $s$. Thus for a constant $s$ the number of clicks from 0 to $\tau$ is a random variable $X$  with Poisson distribution of  parameter $s\tau$. This implies in particulat that $E(X)= s\tau$ and $E(X^2) = s\tau (1+s\tau)$. 

Hence the expectations for $\Delta^h$ are bounded by
$$E(\Delta^h|X^h(0)=x_0)\leq s^h\tau$$
and
$$E((\Delta^h)^2|X^h(0)=x_0)\leq s^h\tau( 1+s^h\tau)$$
for any $x_0\in S(h)$. 
Using (\ref{sh}) we find
$$E(\Delta^h|X^h(0)=x_0)\leq \frac1h \tau(CR^2 +F)$$
and
\begin{align*}
E((\Delta^h)^2|X^h(0)=x_0)&\leq \frac1h \tau(CR^2 +F)[ 1+\frac1h (CR^2 +F)\tau]\\
&= \frac {1}{h^2}\tau(CR^2 +F) [\tau(CR^2 +F) + h].
\end{align*}
Therefore  (\ref{delta1}) and (\ref{delta2}) hold with the  bounds specified in (\ref{i1}) and (\ref{i2}).
% satisfied if and only if $I(h) = \tau$ as in the previous section and $\h \tau(h)= 0$ ; so we define 
%\be
%I_1(h) := \tau(CR^2 +F)
%\label{I1}
%\ee
%\be
%I_2(h) := (CR^2 +F) [ \tau(CR^2 +F) + h]
%\label{I2}
%\ee
%and we have $\h I_1(h) = \h I_2(h) =0$.

\begin{acknowledgement}
 Work (partially) supported by the PhD programme in Mathematical Sciences, Dipartimento di Matematica, Universit\`a di Padova (Italy) and Progetto Dottorati - Fondazione Cassa di Risparmio di Padova e Rovigo.
\end{acknowledgement}

\end{document}